\documentclass[11pt]{article}
\usepackage[T1]{fontenc}
\usepackage{amsfonts}
\usepackage{amsmath}
\usepackage{amssymb}
\usepackage{amsthm}
\usepackage{bbm}
\usepackage{bm}
\usepackage{mathrsfs}
\usepackage{verbatim}
\usepackage{setspace}
\usepackage{color}
\usepackage{pdfsync}
\usepackage{enumitem}

\theoremstyle{plain}
\newtheorem{theorem}{Theorem}[section]
\newtheorem{proposition}[theorem]{Proposition}
\newtheorem{lemma}[theorem]{Lemma}
\newtheorem{corollary}[theorem]{Corollary}

\theoremstyle{definition}

\newtheorem{remark}[theorem]{Remark}

\newtheorem{example}[theorem]{Example}

\newtheorem*{conditionA}{Condition~(A)}
\theoremstyle{remark}

\renewenvironment{thebibliography}[1]{%
\begin{oldthebibliography}{#1}%
\setlength{\baselineskip}{.9em}
\linespread{1}
\small
\setlength{\parskip}{0.3ex}%
\setlength{\itemsep}{.5em}%
}%
{%
\end{oldthebibliography}%
}
\newcommand{\eps}{\varepsilon}

\newcommand{\F}{\mathbb{F}}
\newcommand{\G}{\mathbb{G}}

\newcommand{\N}{\mathbb{N}}

\newcommand{\R}{\mathbb{R}}
\renewcommand{\S}{\mathbb{S}}

\newcommand{\cB}{\mathcal{B}}

\newcommand{\cE}{\mathcal{E}}
\newcommand{\cF}{\mathcal{F}}

\newcommand{\cK}{\mathcal{K}}
\newcommand{\cL}{\mathcal{L}}

\newcommand{\cP}{\mathcal{P}}

\newcommand{\fP}{\mathfrak{P}}

\DeclareMathOperator{\proj}{proj}

\DeclareMathOperator{\LSC}{LSC}
\DeclareMathOperator{\USC}{USC}
\DeclareMathOperator{\SC}{SC}

\DeclareMathOperator{\Lip}{Lip}

\newcommand{\as}{\mbox{-a.s.}}

\newcommand{\1}{\mathbf{1}}

\newcommand{\tomega}{\tilde{\omega}}

\newcommand{\tr}{\mbox{tr}}

\numberwithin{equation}{section}

\usepackage[pdfborder={0 0 0}]{hyperref}
\hypersetup{
  urlcolor = black,
  pdfauthor = {Ariel Neufeld, Marcel Nutz},
  pdfkeywords = {Nonlinear Levy process; Sublinear expectation; Partial integro-differential equation; Semimartingale characteristics; Knightian uncertainty
},
  pdftitle = {Nonlinear Levy Processes and their Characteristics},
  pdfsubject = {Nonlinear Levy Processes and their Characteristics},
  pdfpagemode = UseNone
}

\newcommand{\B}{\mathcal{B}}

\begin{document}

\title{\vspace{0em}
\mbox{Nonlinear L\'evy Processes and their Characteristics}
\date{\today}
\author{
  Ariel Neufeld%
  \thanks{
  Department of Mathematics, ETH Zurich, \texttt{ariel.neufeld@math.ethz.ch}.
  Financial support by Swiss National Science Foundation Grant PDFMP2-137147/1 is gratefully acknowledged.
  }
  \and
  Marcel Nutz%
  \thanks{
  Departments of Statistics and Mathematics, Columbia University, New York, \texttt{mnutz@columbia.edu}. Financial support by NSF Grant DMS-1208985 is gratefully acknowledged.
  }
 }
}
\maketitle \vspace{-1.2em}

\begin{abstract}
  We develop a general construction for nonlinear L\'evy processes with given characteristics. More precisely, given a set $\Theta$ of L\'evy triplets, we construct a sublinear expectation on Skorohod space under which the canonical process has stationary independent increments and a nonlinear generator corresponding to the supremum of all generators of classical L\'evy processes with triplets in $\Theta$. The nonlinear L\'evy process yields a tractable model for Knightian uncertainty about the distribution of jumps for which expectations of Markovian functionals can be calculated by means of a partial integro-differential equation.
\end{abstract}

\vspace{.9em}

{\small

}

\section{Introduction}\label{se:intro}

The main goal of this paper is to construct nonlinear L\'evy processes with prescribed local characteristics. This is achieved by a probabilistic construction involving an optimal control problem on Skorohod space where the controls are laws of semimartingales with suitable characteristics.

Let $X=(X_t)_{t\in\R_+}$ be an $\R^d$-valued process with c\`adl\`ag paths and $X_0=0$, defined on a measurable space $(\Omega,\cF)$ which is equipped with a nonlinear expectation $\cE(\cdot)$. For our purposes, this will be a sublinear operator
\begin{equation}\label{eq:nonlinExpIntro}
  \xi \mapsto \cE(\xi):= \sup_{P\in\fP} E^P[\xi],
\end{equation}
where $\fP$ is a set of probability measures on $(\Omega,\cF)$ and $E^P[\,\cdot\,]$ is the usual expectation, or integral, under the measure $P$. In this setting, if $Y$ and $Z$ are random vectors, $Y$ is said to be independent of $Z$ if
\begin{equation*}
\cE\big(\varphi(Y,Z)\big) = \cE\big(\cE(\varphi(Y,z))|_{z=Z}\big)
\end{equation*}
for all bounded Borel functions $\varphi$, and if $Y$ and $Z$ are of the same dimension, they are said to be identically distributed if
\[
  \cE(\varphi(Y))=\cE(\varphi(Z))
\]
for all bounded Borel functions $\varphi$. We note that both definitions coincide with the classical probabilistic notions if $\fP$ is a singleton.
Following~\cite[Definition~19]{HuPeng.09levy}, the process $X$ is a \emph{nonlinear L\'evy process} under $\cE(\cdot)$ if it has stationary and independent increments; that is, $X_{t}-X_s$ and $X_{t-s}$ are identically distributed for all $0\leq s\leq t$, and $X_{t}-X_s$ is independent of $(X_{s_1},\dots,X_{s_n})$ for all $0\leq s_1 \leq \dots\leq s_n \leq s \leq t$. The particular case of a classical L\'evy process is recovered when $\fP$ is a singleton.

Let $\Theta$ be a set of L\'evy triplets $(b,c,F)$; here $b$ is a vector, $c$ is a symmetric nonnegative matrix, and $F$ is a L\'evy measure. We recall that each L\'evy triplet characterizes the distributional properties and in particular the infinitesimal generator of a classical L\'evy process. More precisely, the associated Kolmogorov equation is
\begin{align*}
  v_t(t,x) - \bigg\{ &b v_x(t,x) + \frac{1}{2} \tr[c v_{xx}(t,x)] \\
  &\!+ \int [v(t,x+z)-v(t,x)-v_x(t,x) h(z)] \, F(dz)\bigg\}=0,
\end{align*}
where, e.g., $h(z)=z\1_{|z|\leq 1}$. Our goal is to construct a nonlinear L\'evy process whose local characteristics are described by the set~$\Theta$, in the sense that the analogue of the Kolmogorov equation will be the fully nonlinear (and somewhat nonstandard) partial integro-differential equation
\begin{align}\label{eq:PIDEintro}
  v_t(t,x) - \!\!\sup_{(b,c,F) \in \Theta}\!\bigg\{ &b v_x(t,x) + \frac{1}{2} \tr[c v_{xx}(t,x)] \\
  &\!+ \int [v(t,x+z)-v(t,x)-v_x(t,x) h(z)] \, F(dz)\bigg\}=0. \nonumber
\end{align}
In fact, our probabilistic construction of the process justifies the name characteristic in a rather direct way.

In our construction, we take $X$ to be the canonical process on Skorohod space and hence $\cE(\cdot)$ is the main object of consideration, or more precisely, the set $\fP$ of probability measures appearing in~\eqref{eq:nonlinExpIntro}. Given an arbitrary set $\Theta$ of L\'evy triplets, we let $\fP=\fP_\Theta$ be the set of all laws of semimartingales whose differential characteristics take values in $\Theta$; that is, their predictable semimartingale characteristics $(B,C,\nu)$ are of the form $(b_t\,dt,c_t\,dt,F_t\,dt)$ and the processes $(b,c,F)$ evolve in $\Theta$. Assuming merely that $\Theta$ is measurable, we then show that $X$ is a nonlinear L\'evy process under $\cE(\cdot)$; this is based on the more general fact that $\cE(\cdot)$ satisfies a certain semigroup property (Theorem~\ref{th:mainLevy}). The proofs require an analysis of semimartingale characteristics which will be useful for other control problems as well. Under the conditions
\begin{equation}\label{eq:condIntro1}
  \sup_{(b,c,F) \in \Theta} \bigg\{\int  |z|\wedge |z|^2 \, F(dz) + |b| + |c| \bigg\}<\infty
\end{equation}
and
\begin{equation}\label{eq:condIntro2}
  \lim_{\eps\to 0} \sup_{(b,c,F) \in \Theta} \int_{|z|\leq \varepsilon} |z|^2 \, F(dz) =0
\end{equation}
on $\Theta$, we show that functionals of the form $v(t,x)=\cE(\psi(x+X_t))$ can be characterized as the unique viscosity solution of the nonlinear Kolmogorov equation~\eqref{eq:PIDEintro} with initial condition $\psi$ (Theorem~\ref{thm:PIDE}).

A special case of a nonlinear L\'evy process with continuous trajectories is called $G$-Brownian motion and due to~\cite{Peng.07, Peng.08}; see also \cite{AvellanedaLevyParas.95,Lyons.95} for the related so-called Uncertain Volatility Model and~\cite{Peng.10} for a monograph providing background related to nonlinear expectations. Nonlinear L\'evy processes were introduced in \cite{HuPeng.09levy}. First, the authors consider a given pair $(X^c,X^d)$ of processes with stationary and independent increments under a given sublinear expectation $\cE(\cdot)$. The continuous process $X^c$ is assumed to satisfy $\cE(|X^c_t|^3)/t\to0$ as $t\to0$ which implies that it is a $G$-Brownian motion, whereas the jump part $X^d$ is assumed to satisfy $\cE(|X^d_t|)\leq Ct$ for some constant $C$.
The sum $X:=X^c+X^d$ is then called a $G$-L\'evy process. It is shown that $G_X[f(\cdot)]:=\lim_{t\to0} \cE(f(X_t))/t$ is well-defined for a suitable class of functions $f$ and has a representation in terms of a set $\Theta$ of L\'evy triplets satisfying
\begin{equation}\label{eq:HuPengCond}
  \sup_{(b,c,F) \in \Theta} \bigg\{\int  |z| \, F(dz) + |b| + |c| \bigg\}<\infty,
\end{equation}
meaning that functions $v(t,x)=\cE(\psi(x+X_t))$ solve the PIDE~\eqref{eq:PIDEintro} with initial condition $\psi$.
We note that~\eqref{eq:HuPengCond} implies both~\eqref{eq:condIntro1} and~\eqref{eq:condIntro2}, and is in fact significantly stronger because it excludes all triplets with infinite variation jumps---the extension of the representation result to such jumps remains an important open problem.
Second, given a set $\Theta$ satisfying~\eqref{eq:HuPengCond}, a corresponding nonlinear L\'evy process $X$ is constructed directly from the PIDE~\eqref{eq:PIDEintro}, in the following sense\footnote{It seems that such a  construction could also be carried out under the weaker conditions~\eqref{eq:condIntro1} and~\eqref{eq:condIntro2}.}. If $X$ is the canonical process, expectations of the form $\cE(\psi(X_t))$ can be defined through the  solution $v$ by simply setting $\cE(\psi(X_t)):=v(t,0)$. More general expectations of the form $\cE(\psi(X_{t_1},\dots, X_{t_n}))$ can be defined similarly by a recursive application of the PIDE. Thus, one can construct $\cE(\xi)$ for all functions $\xi$ in the completion $L^1_G$ of the space of all functions of the form $\psi(X_{t_1},\dots, X_{t_n})$ under the norm $\cE(|\cdot|)$. We remark that this space is significantly smaller than the set of measurable functions, and so it is left open in~\cite{HuPeng.09levy} how to define $\cE(\xi)$ for general random variables $\xi$. See also~\cite{Ren.13} for some properties of~$L^1_G$. A second remark is that while this construction is very direct, it leaves open how to interpret a nonlinear L\'evy process from the point of view of classical probability theory.
Another concept related to nonlinear L\'evy processes is the so-called second order backward stochastic differential equation with jumps as introduced in~\cite{KaziTaniPossamaiZhou.12}, partially extending the formulation of~\cite{SonerTouziZhang.2010bsde} to the case with jumps using an approach along the lines of~\cite{SonerTouziZhang.2010aggreg, SonerTouziZhang.2010dual}. Existence results and the connection to PIDEs have been announced for future work, which is expected to yield another stochastic representation for the PIDE~\eqref{eq:PIDEintro}, under certain conditions.

Summing up, our contribution is twofold. First, we construct nonlinear L\'evy processes for arbitrary (measurable) characteristics $\Theta$, possibly with unbounded diffusion and infinite variation jumps, and the distribution is defined for all measurable functions. Our probabilistic construction allows us to understand the PIDE~\eqref{eq:PIDEintro} as the Hamilton--Jacobi--Bellman equation resulting from the nonstandard control problem
\[
  \sup_{P\in\fP_\Theta} E^P[\,\cdot\,]
\]
over the class of all semimartingales with $\Theta$-valued differential characteristics. This control representation gives a global interpretation to the distribution of a nonlinear L\'evy process as the worst-case expectation over $\fP_\Theta$; in particular, this allows for applications in robust control under model uncertainty (e.g., \cite{NeufeldNutz.12}). Second, under Conditions~\eqref{eq:condIntro1} and~\eqref{eq:condIntro2}, we provide a rigorous link with the PIDE~\eqref{eq:PIDEintro}. On the one hand, this implies that expectations of Markovian functionals can be calculated by means of a differential equation, which is important for applications. On the other hand, it allows us to identify our construction as an extension of~\cite{HuPeng.09levy}. As an example, we introduce a nonlinear version of an $\alpha$-stable L\'evy process which has jumps of infinite variation and thus does not fall within the scope of~\eqref{eq:HuPengCond}.

The remainder of this paper is organized as follows. Section~\ref{se:mainResults} details the setup and contains the main results: the probabilistic construction is summarized in Theorem~\ref{th:mainLevy} and the PIDE characterization in Theorem~\ref{thm:PIDE}. Moreover, we give two examples of nonlinear L\'evy processes. Sections~\ref{se:A2} and~\ref{se:A3} provide an analysis of semimartingale laws and the associated characteristics under conditioning and products which forms the main part of the proof of Theorem~\ref{th:mainLevy} (another ingredient is provided in the companion paper~\cite{NeufeldNutz.13a}). Section~\ref{se:PIDE} concludes with the existence and comparison results for the PIDE~\eqref{eq:PIDEintro}.

\section{Main Results}\label{se:mainResults}

Fix $d\in\N$ and let $\Omega=D_0(\R_+,\R^d)$ be the space of all c\`adl\`ag paths $\omega=(\omega_t)_{t\geq0}$ in $\R^d$ with $\omega_0=0$. We equip $\Omega$ with the Skorohod topology and the corresponding Borel $\sigma$-field $\cF$. Moreover, we denote by $X=(X_t)_{t\geq0}$ the canonical process $X_t(\omega)=\omega_t$ and by $\F=(\cF_t)_{t\geq0}$ the (raw) filtration generated by $X$.

Our starting point is a subset of $\fP(\Omega)$, the Polish space of all probability measures on $\Omega$, determined by the semimartingale characteristics as follows. First, let
\begin{equation}\label{eq:defPsem}
  \fP_{sem}= \{P\in \fP(\Omega)\,|\, \mbox{$X$ is a semimartingale on }(\Omega,\cF,\F,P)\}\subseteq \fP(\Omega)
\end{equation}
be the set of all semimartingale laws. To be specific, let us agree that if $\G$ is a given filtration, a $\G$-adapted process $Y$ with c\`adl\`ag paths will be called a $P$-$\G$-semimartingale if there exist right-continuous, $\G$-adapted processes $M$ and $A$ with $M_0=A_0=0$ such that $M$ is a $P$-$\G$-local martingale, $A$ has paths of (locally) finite variation $P$-a.s., and $Y= Y_0 + M + A$ $P$-a.s. We remark that $X$ is a $P$-semimartingale for $\F$ if and only if it has this property for the right-continuous filtration $\F_+$ or the usual augmentation $\F_+^P$; cf.\ \cite[Proposition~2.2]{NeufeldNutz.13a}. In other words, the precise choice of the filtration in the definition~\eqref{eq:defPsem} is not crucial.

Fix a truncation function $h: \R^d\to \R^d$; that is, a bounded measurable function such that $h(x)=x$ in a neighborhood of the origin, and let $(B^P,C^P,\nu^P)$ be semimartingale characteristics of~$X$ under $P\in \fP_{sem}$ and~$\F$, relative to $h$. To be specific, this means that $(B^P,C^P,\nu^P)$ is a triplet of processes such that $P$-a.s., $B^P$ is the finite variation part in the canonical decomposition of $X -\sum_{0 \leq s \leq \cdot} (\Delta X_s- h(\Delta X_s))$ under $P$, $C^P$ is the quadratic covariation of the continuous local martingale part of~$X$ under $P$, and $\nu^P$ is the $P$-compensator of $\mu^X$, the integer-valued random measure associated with the jumps of $X$. (Again, the precise choice of the filtration does not matter for the present section; see \cite[Proposition~2.2]{NeufeldNutz.13a}.) We shall mainly work with the subset
\begin{equation*}
  \fP^{ac}_{sem}=\big\{P\in \fP_{sem}\,\big|\, \mbox{$(B^P,C^P,\nu^P)\ll dt$, $P$-a.s.}\big\}
\end{equation*}
of semimartingales with absolutely continuous characteristics (with respect to the Lebesgue measure $dt$).
Given $P\in\fP^{ac}_{sem}$, we can consider the associated differential characteristics $(b^P,c^P, F^P)$, defined via $(dB^P,dC^P,d\nu^P)=(b^P dt, c^Pdt, F^Pdt)$. The differential characteristics take values in $\R^d \times \S^d_+\times\cL$, where $\S^d_+$ is the set of symmetric nonnegative definite $d\times d$-matrices and
\begin{equation*}
  \cL= \bigg\{ F \mbox{ measure on } \R^d \, \bigg| \,  \int_{\R^d} |x|^2 \wedge 1 \, F(dx) <\infty \ \mbox{and} \ F(\{0\})=0  \bigg\}
\end{equation*}
is the set of all L\'evy measures, a separable metric space under a suitable version of the weak convergence topology (cf.\ \cite[Section~2]{NeufeldNutz.13a}). Any element $(b,c,F)\in \R^d \times \S^d_+\times\cL$ is called a L\'evy triplet and indeed, there exists a L\'evy process having $(b,c,F)$ as its differential characteristics.

\subsection{Nonlinear L\'evy Processes with given Characteristics}

Let $\emptyset\neq\Theta\subseteq \R^d \times \S^d_+\times\cL$ be any (Borel) measurable subset. Our aim is to construct a nonlinear L\'evy process corresponding to $\Theta$; of course, the case of a classical L\'evy process will correspond to $\Theta$ being a singleton. An important object in our construction is the set of all semimartingale laws whose differential characteristics take values in $\Theta$,
\[%
  \fP_\Theta:=\big\{P\in \fP^{ac}_{sem}\,\big|\, (b^P,c^P,F^P)\in \Theta, \,P\otimes dt\mbox{-a.e.}\big\},
\]%
and a key step will be to show that $\fP_\Theta$ is amenable to dynamic programming, as formalized by Condition~(A) below. To state this condition, we need to introduce some more notation.
Let $\tau$ be a finite $\F$-stopping time. Then
the concatenation of $\omega, \tilde{\omega}\in \Omega$ at $\tau$ is the path
\[
 (\omega\otimes_\tau \tilde{\omega})_u := \omega_u \1_{[0,\tau(\omega))}(u) + \big(\omega_{\tau(\omega)} + \tilde{\omega}_{u-\tau(\omega)}\big) \1_{[\tau(\omega), \infty)}(u),\quad u\geq 0.
\]
For any probability measure $P\in\fP(\Omega)$, there is a regular conditional
probability distribution $\{P^\omega_\tau\}_{\omega\in\Omega}$
given $\cF_\tau$ satisfying
\[%
  P^\omega_\tau\big\{\omega'\in \Omega \,\big|\, \omega' = \omega \mbox{ on } [0,\tau(\omega)]\big\} = 1\quad\mbox{for all}\quad\omega\in\Omega.
\]%
We then define $P^{\tau,\omega}\in \fP(\Omega)$ by
\[
  P^{\tau,\omega}(D):=P^\omega_\tau(\omega\otimes_\tau D),\quad D\in \cF, \quad\mbox{where }\omega\otimes_\tau D:=\{\omega\otimes_\tau \tilde{\omega}\,|\, \tilde{\omega}\in D\}.
\]
Given a function $\xi$ on $\Omega$ and $\omega\in\Omega$,
we also define the function $\xi^{\tau,\omega}$ on $\Omega$ by
\[
  \xi^{\tau,\omega}(\tilde{\omega}) :=\xi(\omega\otimes_\tau \tilde{\omega}),\quad \tilde{\omega}\in\Omega.
\]
If $\xi$ is measurable, we then have $E^{P^{\tau,\omega}}[\xi^{\tau,\omega}]=E^P[\xi|\cF_\tau](\omega)$ for $P$-a.e.\ $\omega\in\Omega$. (The convention $\infty - \infty = -\infty$ is used throughout; for instance, in defining $E^P[\xi|\cF_\tau]:=E^P[\xi^+|\cF_\tau]-E^P[\xi^-|\cF_\tau]$.)
Finally, a subset of a Polish space is called analytic if it is the image of a Borel subset of another Polish space under a Borel-measurable mapping; in particular, any Borel set is analytic.

We can now state the mentioned condition for a given set $\fP \subseteq \fP(\Omega)$.

\begin{conditionA}\label{condA}
  Let $\tau$ be a finite $\F$-stopping time  and let  $P\in \fP$.
  \begin{enumerate}%
    \item[(A1)] The set $\fP \subseteq\, \fP(\Omega)$ is analytic.

    \item[(A2)] We have $P^{\tau,\omega} \in\fP$ for $P$-a.e.\ $\omega\in\Omega$.

    \item[(A3)] If $\kappa: \Omega \to \fP(\Omega)$ is an $\cF_\tau$-measurable kernel and $\kappa(\omega)\in \fP$ for $P$-a.e.\ $\omega\in\Omega$,
    then the measure defined by
    \[%
      \bar{P}(D)=\iint (\1_D)^{\tau,\omega}(\omega') \,\kappa(\omega,d\omega')\,P(d\omega),\quad D\in \cF
    \]%
    is an element of $\fP$.
  \end{enumerate}
\end{conditionA}

Some more notation is needed for the first main result. Given a $\sigma$-field $\mathcal{G}$, the universal completion of $\mathcal{G}$ is the $\sigma$-field $\mathcal{G}^*=\cap_P \mathcal{G}^{(P)}$, where $P$ ranges over all probability measures on $\mathcal{G}$ and $\mathcal{G}^{(P)}$ is the completion of $\mathcal{G}$ under $P$.  Moreover, an $\overline{\R}$-valued function $f$ is called upper semianalytic if $\{f>a\}$ is analytic for each $a\in\R$. Any Borel-measurable function is upper semianalytic and any upper semianalytic function is universally measurable.

\begin{theorem}\label{th:mainLevy}
  Let $\Theta\subseteq \R^d \times \S^d_+\times\cL$ be a measurable set of L\'evy triplets, $\fP_\Theta=\big\{P\in \fP^{ac}_{sem}\,\big|\, (b^P,c^P,F^P)\in \Theta, P\otimes dt\mbox{-a.e.}\big\}$ and consider the associated sublinear expectation $\cE(\cdot)=\sup_{P\in\fP_\Theta} E^P[\,\cdot\,]$ on the Skorohod space~$\Omega$.
  \begin{enumerate}
   \item The set $\fP_\Theta$ satisfies Condition~(A).
   \item Let $\sigma\leq\tau$ be finite $\F$-stopping times and let $\xi:\Omega\to\overline{\R}$ be upper semianalytic. Then the function
   $\omega\mapsto \mathcal{E}_\tau(\xi)(\omega):=\cE(\xi^{\tau,\omega})$
   is $\cF_\tau^*$-measurable and upper semianalytic, and
   \begin{equation}\label{eq:DPP}
     \mathcal{E}_\sigma(\xi)(\omega) = \mathcal{E}_\sigma(\mathcal{E}_\tau(\xi))(\omega)\quad\mbox{for all}\quad \omega\in\Omega.
   \end{equation}
   \item The canonical process $X$ is a nonlinear L\'evy process under $\cE(\cdot)$.
  \end{enumerate}
\end{theorem}

\pagebreak[4]

Thus, this results yields the existence of nonlinear L\'evy processes with general characteristic $\Theta$ as well as their interpretation in terms of classical stochastic analysis; namely, as a control problem over laws of semimartingales. The semigroup property stated in~\eqref{eq:DPP} will be the starting point for the PIDE result reported below.

\begin{proof}
  (i) The verification of (A1) is somewhat lengthy and carried out in the companion paper~\cite{NeufeldNutz.13a}. There, it is shown that one can construct a version of the semimartingale characteristics which is measurable with respect to the underlying measure $P$, and this fact is used to show that $\fP_\Theta$ is Borel-measurable (and in particular analytic); cf.\ \cite[Corollary~2.7]{NeufeldNutz.13a}. Properties~(A2) and~(A3) will be established in Corollary~\ref{co:A2} and Proposition~\ref{prop:A3}, respectively. They follow from the analysis of semimartingale characteristics under conditioning and products of semimartingale laws that will be carried out in Sections~\ref{se:A2} and~\ref{se:A3}.

  (ii) Once Condition~(A) is established, the validity of~(ii) is a consequence of the dynamic programming principle in the form of~\cite[Theorem~2.3]{NutzVanHandel.12}. (That result is stated for the space of continuous paths, but carries over to Skorohod space with the same proof.)

  (iii) We first show that $X$ has stationary increments. Let $s,t\geq0$ and let $\varphi: \R^d\to\R$ be bounded and Borel. Using the identity
  \[
    X^{t,\omega}_{t+s}-X^{t,\omega}_t = X_s,\quad \omega\in\Omega,
  \]
  the tower property~\eqref{eq:DPP} yields that
  \begin{align*}
    \cE\big(\varphi(X_{t+s}-X_t)\big)
    = \cE\big(\cE_t\big(\varphi(X_{t+s}-X_t)\big)\big)
    =  \cE\big(\cE\big(\varphi(X_s)\big)\big)
    = \cE\big(\varphi(X_s)\big).
  \end{align*}
  Similarly, to see the independence of the increments, let $0\leq t_1 \leq\dots\leq t_n \leq t$ and let $\varphi$ be defined on $\R^{(n+1)d}$ instead of $\R^d$. Then
  \[
    \big(X^{t,\omega}_{t+s}-{X}^{t,\omega}_t,X^{t,\omega}_{t_1},\dots,X^{t,\omega}_{t_n}\big) = \big(X_s,X_{t_1}(\omega),\dots,X_{t_n}(\omega)\big),\quad \omega\in\Omega
  \]
  and~\eqref{eq:DPP} imply that
  \begin{align*}
  & \ \cE\big(\varphi(X_{t+s}-{X}_t,X_{t_1},\dots,X_{t_n})\big) \\
  = & \ \cE\big(\cE_t\big(\varphi(X_{t+s}-X_t,X_{t_1},\dots,X_{t_n})\big)\big)\\
  = & \ \cE\big(\cE\big(\varphi(X_s,x_1,\dots,x_n)\big)|_{x_1=X_{t_1},\dots,
  x_n=X_{t_n}} \big)\\
  = & \ \cE\big(\cE\big(\varphi(X_{t+s}-{X}_t,x_1,\dots,x_n)\big)|_{x_1=X_{t_1},\dots,
  x_n=X_{t_n}} \big),
  \end{align*}
  where the last equality is due to the stationarity of the increments applied to the test function $\varphi(\cdot,x_1,\dots,x_n)$.
\end{proof}

\begin{remark}
  The nonlinear L\'evy property of $X$ corresponds to the fact that the set $\fP=\fP_\Theta$ is independent of $(t,\omega)$. More precisely, recall that~\cite{NutzVanHandel.12} considered more generally a family $\{\fP(t,\omega)\}$ indexed by $t\geq0$ and $\omega\in\Omega$. In this situation, the conditional nonlinear expectation is given by
  \[
    \mathcal{E}_\tau(\xi)(\omega):=\sup_{P\in\fP(\tau(\omega),\omega)} E^P[\xi^{\tau,\omega}],\quad\omega\in\Omega;
  \]
  this coincides with the above definition when $\fP(t,\omega)$ is independent of $(t,\omega)$. As can be seen from the above proof, the temporal and spatial homogeneity of $\fP$ is essentially in one-to-one correspondence with the independence and stationarity of the increments of $X$ under $\cE(\cdot)$.
\end{remark}

In classical stochastic analysis, L\'evy processes can be characterized as semimartingales with constant differential characteristics. The following shows that the nonlinear case allows for a richer structure.

\begin{remark}
  The assertion of Theorem~\ref{th:mainLevy} holds more generally for any set $\fP\subseteq \fP(\Omega)$ satisfying Condition~(A); this is clear from the proof. According to Theorem~\ref{th:charactinv},  Proposition~\ref{prop:stabsem} and~\cite[Theorem~2.5]{NeufeldNutz.13a}, the collection $\fP_{sem}$ of all semimartingale laws (not necessarily with absolutely continuous characteristics) is another example of such a set. In particular, we see that nonlinear L\'evy processes are not constrained to the time scale given by the Lebesgue measure.
  It is well known that classical L\'evy processes have this property and one may say that this is due to the fact that the Lebesgue measure is, up to a normalization, the only homogeneous (shift-invariant) measure on the line. By contrast, there are many sublinear expectations on the line that are homogeneous---for instance, the one determined by the supremum of all measures, which may be seen as the time scale corresponding to $\fP_{sem}$.

  Another property of classical L\'evy processes is that they are necessarily semimartingales. A trivial example satisfying Condition~(A) is the set $\fP=\fP(\Omega)$ of all probability measures on $\Omega$. Thus, we also see that the semimartingale property, considered under a given $P\in\fP$, does not hold automatically.

  One may also note that such (degenerate) examples are far outside the scope of the PIDE-based construction of~\cite{HuPeng.09levy}.
\end{remark}

\begin{remark}
  The present setup could be extended to a case where the set $\Theta$ is replaced by a set-valued process $(t,\omega)\mapsto \Theta(t,\omega)$, in the spirit of the random $G$-expectations~\cite{Nutz.10Gexp}. Of course, this situation is no longer homogeneous and so the resulting process would be a ``nonlinear semimartingale'' rather than a L\'evy process. We shall see in the subsequent sections that the techniques of the present paper still yield the desired dynamic programming properties, exactly as it was done in~\cite{NutzVanHandel.12} for the case of continuous martingales.
\end{remark}

\subsection{Nonlinear L\'evy Processes and PIDE}

For the second main result of this paper, consider a nonempty measurable set $\Theta \subseteq \R^d \times \S_{+}^d \times \mathcal{L}$ satisfying the following two additional assumptions. The first one is
\begin{equation}\label{eq:intcond}
  \sup_{(b,c,F) \in \Theta} \bigg\{\int_{\R^d} |z|\wedge |z|^2 \, F(dz) + |b| + |c| \bigg\}<\infty,
\end{equation}
where $|\cdot|$ is the Euclidean norm; this implies that the control problem defining $\cE(\cdot)$ is non-singular and that the jumps are integrable. Moreover, we require that
\begin{equation}\label{eq:limitcond}
\lim_{\eps\to 0} \sup_{(b,c,F) \in \Theta} \int_{|z|\leq \varepsilon} |z|^2 \, F(dz) =0.
\end{equation}
While this condition does not exclude any particular L\'evy measure, it bounds the contribution of small jumps across $\Theta$. In particular, it prevents $\fP_\Theta$ from containing a sequence of pure-jump processes which converges weakly to, say, a Brownian motion. %
Thus, both conditions are necessary to ensure that the PIDE below is indeed the correct dynamic programming equation for our problem. %

Namely, we fix $\psi \in C_{b,Lip}(\R^d)$, the space of bounded Lipschitz functions on $\R^d$, and consider the fully nonlinear PIDE
\begin{equation}\label{eq:PIDE}
  \begin{cases}
    \partial_t v(t,x)-G\big(D_x v(t,x), D^2_{xx} v(t,x), v(t,x+\cdot)\big)=0 \quad \mbox{on }(0,\infty) \times \R^d, \\
    \quad v(0,\cdot)=\psi(\cdot),
  \end{cases}
\end{equation}
where $G:\R^d \times \S^d \times C_b^{2}(\R^d) \to \R$ is defined by
\begin{multline}\label{eq:PIDE-G}
G\big(p,q,f(\cdot)\big) \\
 := \ \sup_{(b,c,F) \in \Theta}\bigg\{p b + \frac{1}{2} \tr[q c] + \int_{\R^d} \big[f(z)-f(0)- D_x f(0) h(z)\big] F(dz) \bigg\}.
\end{multline}
We remark that this PIDE is nonstandard due to the supremum over a set of L\'evy measures; see also~\cite{HuPeng.09levy}. Specifically, since this set is typically large (nondominated),~\eqref{eq:PIDE-G} does not satisfy a dominated convergence theorem with respect to $f$, which leads to a discontinuous operator $G$.

We write $C_b^{2,3}((0,\infty)\times \R^d)$ for the set of functions on $(0,\infty)\times \R^d$ having bounded continuous derivatives up to the second and third order in $t$ and $x$, respectively. A bounded upper semicontinuous function $u$ on $[0,\infty)\times \R^d$ will be called a viscosity subsolution of~\eqref{eq:PIDE} if $u(0,\cdot)\leq \psi(\cdot)$ and
\begin{equation*}
\partial_t \varphi(t,x)-G\big(D_x \varphi(t,x), D^2_{xx} \varphi(t,x), \varphi(t,x+\cdot)\big)\leq 0
\end{equation*}
whenever $\varphi \in C_b^{2,3}((0,\infty)\times \R^d)$ is such that $\varphi\geq u$ on $(0,\infty)\times \R^d$ and $\varphi(t,x)= u(t,x)$ for some $(t,x) \in (0,\infty)\times \R^d$.
The definition of a viscosity supersolution is obtained by reversing the inequalities and the semicontinuity. Finally, a bounded continuous function is a viscosity solution if it is both sub- and supersolution. We recall that $\cE(\cdot)=\sup_{P\in\fP_\Theta} E^P[\,\cdot\,]$ and $X$ is the canonical process.

\begin{theorem}\label{thm:PIDE}
  Let $\Theta \subseteq \R^d \times \S_{+}^d \times \mathcal{L}$ be a measurable set satisfying~\eqref{eq:intcond} and~\eqref{eq:limitcond} and let $\psi \in C_{b,Lip}(\R^d)$.
  Then
  \begin{equation} \label{eq:val}
    v(t,x):=\cE(\psi(x+X_{t})\big), \quad (t,x) \in [0,\infty)\times \R^d
  \end{equation}
  is the unique viscosity solution of~\eqref{eq:PIDE}.
\end{theorem}

The existence part will be proved in Proposition~\ref{prop:PIDEex}, whereas the validity of a comparison principle (and thus the uniqueness) is obtained in Proposition~\ref{pr:comparisonPIDE}. As mentioned in the Introduction, this result allows us to rigorously identify our construction as an extension of~\cite{HuPeng.09levy}. A quite different application is given in Example~\ref{ex:stable} below.

\subsection{Examples}\label{se:examples}

We conclude this section with two examples of nonlinear L\'evy processes in dimension $d=1$. The first one, called Poisson process with uncertain intensity, is the simplest example of interest and was already introduced in~\cite{HuPeng.09levy} under slightly more restrictive assumptions.

\begin{example}
  Fix a measurable set $\Lambda\subseteq \R_+$ and consider
  \[
    \Theta:=\big\{(0,0,\lambda \delta_{1}(dx)) \, \big|  \, \lambda \in \Lambda \big\}.
  \]
  Each triplet in $\Theta$ corresponds to a Poisson process with some intensity $\lambda\in\Lambda$, so that $\Lambda$ can be called the set of possible intensities.
  To see that $\Theta$ is measurable, note that $\Theta$ is the image of $\Lambda$ under $\lambda\mapsto \lambda \, \delta_{1}(dx)$. This is a measurable one-to-one mapping from $\R_+$ into $\cL$, and as $\cL$ is a separable metric space according to \cite[Lemma~2.3]{NeufeldNutz.13a}, it follows by Kuratowski's theorem \cite[Proposition~7.15, p.\,121]{BertsekasShreve.78} that $\Theta$ is indeed measurable.

  As a result, Theorem~\ref{th:mainLevy} shows that the canonical process $X$ is a nonlinear L\'evy process with respect to $\cE(\cdot)=\sup_{P \in \fP_\Theta}E^P[\,\cdot\,]$. Moreover, if $\Lambda$ is bounded, Conditions~\eqref{eq:intcond} and~\eqref{eq:limitcond} hold and Theorem~\ref{thm:PIDE} yields that $v(t,x):= \cE(\psi(x+ X_t))$ is the unique viscosity solution of the PIDE~\eqref{eq:PIDE} with nonlinearity
  \[
    G\big(p,q,f(\cdot)\big)
   = \sup_{\lambda \in \Lambda} \lambda \,\big(f(1)-f(0)- D_x f(0)h(1)\big),
  \]
  for all $\psi \in C_{b,Lip}(\R)$.
\end{example}

The second example represents uncertainty over a family of stable triplets, which does not fall within the framework of~\cite{HuPeng.09levy} because of the infinite variation jumps.  We shall exploit the PIDE result to infer a nontrivial distributional property. In view of the central limit theorem of~\cite{Peng.10CLT} for the nonlinear Gaussian distribution and classical results for $\alpha$-stable distributions, one may suspect that this example also yields the limiting distribution in a nonstandard limit theorem\footnote{Such a result was indeed obtained in follow-up work~\cite{BayraktarMunk.14}.}.

\begin{example}\label{ex:stable}
   Let $\alpha\in(0,2)$, fix  measurable sets $B\subseteq\R$ and $K_\pm\subseteq\R_+$, and consider
   \[
     \Theta :=\big\{(b,0,F_{k_\pm}) \, \big|  \, b \in B,\,k_{\pm}\in K_{\pm}\big\},
   \]
   where $F_{k_\pm}$ denotes the $\alpha$-stable L\'evy measure
   \[
     F_{k_\pm}(dx) = \big(k_{-} \, \1_{(-\infty,0)} + k_{+} \, \1_{(0,\infty)}\big)(x) \, |x|^{-\alpha-1} \, dx.
   \]
   If $f$ is a bounded measurable function on $\R$, then $(k_+,k_-)\mapsto \int f(x)\,F_{k_\pm}(dx)$ is measurable by Fubini's theorem. In view of \cite[Lemma~2.4]{NeufeldNutz.13a}, this means that
   $(k_+,k_-)\mapsto F_{k_\pm}$ is a measurable one-to-one mapping into $\cL$, and thus Kuratowski's theorem again yields that $\Theta$ is measurable.

   As a result, Theorem~\ref{th:mainLevy} once more shows that the canonical process $X$ is a nonlinear L\'evy process with respect to $\cE(\cdot)=\sup_{P \in \fP_\Theta}E^P[\,\cdot\,]$. If $B,K_\pm$ are bounded and $\alpha\in(1,2)$, Conditions~\eqref{eq:intcond} and~\eqref{eq:limitcond} hold and Theorem~\ref{thm:PIDE} yields that $v(t,x):= \cE(\psi(x+ X_t))$ is the unique viscosity solution of the PIDE~\eqref{eq:PIDE} with
   \[
    G\big(p,q,f(\cdot)\big)
     = \sup_{b\in B,\,k_{\pm}\in K_{\pm}} \bigg\{
     p b + \int_{\R} \big(f(z)-f(0)- D_x f(0) h(z)\big) \, F_{k_\pm}(dz)
     \bigg\},
    \]
  for all $\psi \in C_{b,Lip}(\R)$.

  With these conditions still in force, we now use the PIDE to see that $X$ indeed satisfies a scaling property like the classical stable processes; namely, that $X_{\lambda t}$ and $\lambda^{1/\alpha}X_t$ have the same distribution in the sense that
  \[
    \cE(\psi(X_{\lambda t}))=\cE(\psi(\lambda^{1/\alpha}X_t)),\quad \psi \in C_{b,Lip}(\R)
  \]
  for all $\lambda>0$ and $t\geq0$, provided that $X_t$ is centered. More precisely, as $\alpha\in(1,2)$, we may state the characteristics with respect to $h(x)=x$. In this parametrization, we suppose that $B=\{0\}$, since clearly no scaling property can exist in the situation with drift uncertainty. Given $\psi \in C_{b,Lip}(\R)$, Theorem~\ref{thm:PIDE} yields that $\cE(\psi(X_{\lambda t}))=v(\lambda t,0)$, where $v$ is the unique solution of the PIDE with initial condition $\psi$. If we define $\tilde{v}(t,x):=v(\lambda t,\lambda^{1/\alpha}x)$, it follows from
  \[
    G\big(p,q,f(\lambda^{1/\alpha}\cdot)\big)=\lambda G\big(p,q,f(\cdot)\big),\quad f\in C_b^{2}(\R)
  \]
  that $\tilde{v}$ is the (unique) viscosity solution to the same PIDE with initial condition $\tilde{\psi}(x):=\psi(\lambda^{1/\alpha}x)$. In particular, $\tilde{v}(t,0)=\cE(\tilde{\psi}(X_t))$ by Theorem~\ref{thm:PIDE}. As a result, we have
  \[
    \cE(\psi(X_{\lambda t})) = v(\lambda t,0)=\tilde{v}(t,0)=\cE(\tilde{\psi}(X_t))=\cE(\psi(\lambda^{1/\alpha}X_t))
  \]
  as claimed. Note that  $\fP_\Theta$ contains many semimartingale laws  which do not satisfy the scaling property, so that this identity is indeed not trivial.
\end{example}

\section{Conditioned Semimartingale Laws and~(A2)}\label{se:A2}

In this section, we show that given $P\in\fP_{sem}$, the measures of the form $P^{\tau,\omega}$ are again semimartingale laws, and we establish the corresponding transformation of the semimartingale characteristics.  In particular, this will yield the property~(A2) for the set $\fP_\Theta$ as required by the main results.

We remark that the use of the raw filtration $\F$ has some importance in this section; for instance, we shall frequently apply Galmarino's test and related properties.
The following notation will be used. Let $P\in\fP_{sem}$ and let $\nu(\cdot,dt,dz)$ be the $P$-$\F$-compensator of $\mu^X$; that is, the third characteristic under $P$. Then there exists a decomposition
\begin{equation}\label{eq:nuDecomp}
  \nu(\cdot,dt,dx)=F_{\cdot,t}(dz)\,dA_t(\cdot)\quad P\as,
\end{equation}
where $F_{\cdot,t}(dz)$ is a kernel from $(\Omega\times[0,\infty),\mathcal{P})$ into $(\R^d,\mathcal{B}(\R^d))$ and $A$ is an $\F$-predictable process with $A_0=0$ and $P$-a.s.\ non-decreasing, $P$-a.s.\ right-continuous paths; cf.\ \cite[Theorem~II.1.8, p.\,66]{JacodShiryaev.03} and \cite[Lemma~7, p.\,399]{DellacherieMeyer.82}.
We often write $F_{t}(dz)$ instead of $F_{\cdot,t}(dz)$. Moreover, if $Y$ is a stochastic process and $\sigma,\tau$ are finite stopping times, we simply write $Y^{\tau,\omega}_{\sigma+\cdot}$ for $(Y_{\sigma+\cdot})^{\tau,\omega}$; that is, the process $(\tilde\omega,t)\mapsto Y_{\sigma(\omega\otimes_\tau \tilde\omega) + t} (\omega\otimes_\tau \tilde\omega)$.

\begin{theorem}\label{th:charactinv}
  Let $P \in \fP_{sem}$, let $\tau$ be a finite $\F$-stopping time, and let $(B,C,F(dz)\,dA)$ be $P$-$\F$-characteristics of $X$. For $P$-a.e.\ $\omega\in\Omega$, we have $P^{\tau,\omega}\in \fP_{sem}$ and the processes
  \[%
    B^{\tau,\omega}_{\tau+ \cdot}-B_{\tau(\omega)}(\omega),\quad C^{\tau,\omega}_{\tau+ \cdot}-C_{\tau(\omega)}(\omega),\quad F^{\tau,\omega}_{\tau+\cdot}(dz)\,d(A_{\tau+\cdot}^{\tau,\omega}-A_{\tau(\omega)}(\omega))
  \]%
  define a triplet of $P^{\tau,\omega}$-$\F$-characteristics of $X$. Moreover, if $P\in\fP^{ac}_{sem}$ and $(b,c,F)$ are differential $P$-$\F$-characteristics, then for $P$-a.e.\ $\omega\in\Omega$, we have $P^{\tau,\omega}\in \fP^{ac}_{sem}$ and
  \[%
    b^{\tau,\omega}_{\tau+\cdot}, \quad c^{\tau,\omega}_{\tau+\cdot}, \quad F^{\tau,\omega}_{\tau+\cdot}(dz)
  \]%
  define differential $P^{\tau,\omega}$-$\F$-characteristics of $X$.
\end{theorem}

The proof will be given in the course of this section. Before that, let us state a consequence which forms part of Theorem~\ref{th:mainLevy}.

\begin{corollary}\label{co:A2}
  Let $\Theta \subseteq \R^d \times \S^d_{+}  \times \mathcal{L}$ be measurable, let $P \in \fP_{\Theta}$ and let $\tau$ be a finite $\F$-stopping time. Then $P^{\tau,\omega} \in \fP_{\Theta}$ for $P$-a.e.\ $\omega \in \Omega$; that is, $\fP_{\Theta}$ satisfies~(A2).
\end{corollary}

\begin{proof}
  This is a direct consequence of the formula for the differential characteristics under $P^{\tau,\omega}$ from Theorem~\ref{th:charactinv}  and the definition of $\fP_\Theta$.
\end{proof}

As a first step towards the proof of Theorem~\ref{th:charactinv}, we establish two facts about the conditioning of (local) martingales.

\begin{lemma}\label{le:marttransinv}
  Let $P \in \fP(\Omega)$, let $M$ be a $P$-$\F$-uniformly integrable martingale with right-continuous paths and let $\tau$ be a finite $\F$-stopping time. Then $M^{\tau,\omega}_{\tau+\cdot}$ is a $P^{\tau,\omega}$-$\F$-martingale for $P$-a.e.\ $\omega \in \Omega$.
\end{lemma}

\begin{proof}
  By Galmarino's test, $M^{\tau,\omega}_{\tau+\cdot}$ is $\F$-adapted. Moreover,
  \begin{equation*}
    E^{P^{\tau,\omega}}\big[|M^{\tau,\omega}_{\tau+t}|\big]= E^{P}\big[|M_{\tau+t}|\,\big| \, \cF_\tau \big](\omega)<\infty  \ \ \mbox{for} \quad P \mbox{-a.e.} \ \omega \in \Omega
  \end{equation*}
  and all $t\geq 0$. Let $0\leq u\leq v<\infty$ and let $f$ be a bounded $\cF_u$-measurable function. Define the function $\hat{f}$ by
  \[
    \hat{f}(\omega):= f(\omega_{\tau(\omega)+\cdot}-\omega_\tau(\omega)),\quad \omega\in\Omega;
  \]
  then $\hat{f}$ is $\cF_{\tau+u}$-measurable and ${\hat{f}\,}^{\tau,\omega}=f$. Applying the optional sampling theorem to the right-continuous, $P$-$\F$-uniformly integrable martingale $M$, we obtain that
  \begin{equation*}
    E^{P^{\tau,\omega}}\big[\big(M^{\tau,\omega}_{\tau+v}-M^{\tau,\omega}_{\tau+u}\big)\,f \big]
    =  E^{P}\big[\big(M_{\tau+v}-M_{\tau+u}\big)\,\hat{f}\,\big| \, \cF_\tau \big](\omega)
    =0
  \end{equation*}
  for $P$-a.e.\ $\omega \in \Omega$. This implies the martingale property of $M^{\tau,\omega}_{\tau+\cdot}$ as claimed.
\end{proof}

\begin{lemma}\label{le:localmarttransinv}
  Let $P \in \fP(\Omega)$, let $M$ be a right-continuous $P$-$\F$-local martingale having $P$-a.s.\ c\`adl\`ag paths and uniformly bounded jumps, and let $\tau$ be a finite $\F$-stopping time. Then $M^{\tau,\omega}_{\tau+\cdot}$ is a $P^{\tau,\omega}$-$\F$-local martingale for $P$-a.e.\ $\omega \in \Omega$.
\end{lemma}

\begin{proof}
  Again, $M^{\tau,\omega}_{\tau+\cdot}$ is $\F$-adapted and has right-continuous paths for any $\omega \in \Omega$. Let $(T_m)_{m \in \N}$ be a localizing sequence of the $P$-$\F$-local martingale $M$ such that $T_m \leq m$. Since $T_m\to\infty$ $P$-a.s., we have that $T^{\tau,\omega}_m\to\infty$ $P^{\tau,\omega}$-a.s.\ for $P$-a.e.\ $\omega \in \Omega$. Moreover, by Lemma~\ref{le:marttransinv}, each process $M^{\tau,\omega}_{T_m\wedge(\tau+\cdot)}$ is a $P^{\tau,\omega}$-$\F$-martingale for $P$-a.e.\ $\omega \in \Omega$.

  Thus, there exists a subset $\Omega'\subseteq \Omega$ of full $P$-measure such that for all $\omega \in \Omega'$, we have the following three properties:
  $M^{\tau,\omega}_{\tau+\cdot}$ has c\`adl\`ag paths with uniformly bounded jumps $P^{\tau,\omega}$-a.s.,
  the process $M^{\tau,\omega}_{T_m\wedge(\tau+\cdot)}$  is a $P^{\tau,\omega}$-$\F$-martingale for all $m \in \N$,
  and $T^{\tau,\omega}_m\to \infty$ $P^{\tau,\omega}$-a.s. In what follows, we fix $\omega\in\Omega'$ and show that $M^{\tau,\omega}_{\tau+\cdot}$ is a $P^{\tau,\omega}$-$\F$-local martingale.
  Define
  \begin{align*}
  \rho_n&:=\inf\big\{ t\geq 0 \, \big| \, |M_{\tau+t}|\geq n \mbox{ or } |M_{\tau+t-}|\geq n \}\wedge n.
  \end{align*}
  Using that $M^{\tau,\omega}_{\tau+\cdot}$ has c\`adl\`ag paths $P^{\tau,\omega}$-a.s., we see that $\rho_n^{\tau,\omega}$ is a stopping time of $\F^{P^{\tau,\omega}}$, the augmentation of $\F$ under $P^{\tau,\omega}$, and that $\rho_n^{\tau,\omega}\to\infty$ $P^{\tau,\omega}$-a.s.
  Since $M^{\tau,\omega}_{\tau+\cdot}$ has uniformly bounded jumps $P^{\tau,\omega}$-a.s., we have that
  \[
    E^{P^{\tau,\omega}}\Big[ \sup_{t \geq 0} \big|M^{\tau,\omega}_{\tau+(\rho_n\wedge t)}\big|\Big]
   \leq n + E^{P^{\tau,\omega}}\Big[\big|\Delta M^{\tau,\omega}_{\tau+\rho_n}|\Big]< \infty
  \]
  for all $n$. Therefore, given $0 \leq u \leq v <\infty$, the dominated convergence theorem and the optional sampling theorem applied to the martingale $M^{\tau,\omega}_{T_m\wedge(\tau+\cdot)}$ and the stopping time $\rho_n^{\tau,\omega}$ yield that
 \begin{align*}
   E^{P^{\tau,\omega}}\Big[M^{\tau,\omega}_{\tau+(\rho_n\wedge v)} \, \Big| \, \cF^{P^{\tau,\omega}}_u \Big]
   &= E^{P^{\tau,\omega}}\Big[ \lim\limits_{m \to \infty} M^{\tau,\omega}_{T_m\wedge(\tau+(\rho_n\wedge v))}\, \Big| \, \cF^{P^{\tau,\omega}}_u\Big]\\
   &=\lim\limits_{m \to \infty} E^{P^{\tau,\omega}}\Big[  M^{\tau,\omega}_{T_m\wedge(\tau+(\rho_n\wedge v))}\, \Big| \, \cF^{P^{\tau,\omega}}_u\Big]\\
   &= \lim\limits_{m \to \infty} M^{\tau,\omega}_{T_m\wedge(\tau+(\rho_n\wedge u))}\\
   &= M^{\tau,\omega}_{\tau+(\rho_n\wedge u)}\quad P^{\tau,\omega}\as
  \end{align*}
  Thus, $M^{\tau,\omega}_{\tau+(\rho_n\wedge \cdot)}$ is a $P^{\tau,\omega}$-$\F^{P^{\tau,\omega}}$-uniformly integrable martingale for each $n \in \N$, meaning that $M^{\tau,\omega}_{\tau+\cdot}$ is a $P^{\tau,\omega}$-$\F^{P^{\tau,\omega}}$-local martingale, localized by the $\F^{P^{\tau,\omega}}$-stopping times $\rho_n^{\tau,\omega}$.

  It remains to return to the original filtration $\F$. Indeed, we first note that by a standard backward martingale convergence argument, $M^{\tau,\omega}_{\tau+\cdot}$ is also a $P^{\tau,\omega}$-$\F_+^{P^{\tau,\omega}}$-local martingale; cf.\ \cite[Lemma~II.67.10, p.\,173]{RogersWilliams.94}.
  It then follows from \cite[Theorem~3]{Dellacherie.78} that there exists an $\F_+^{P^{\tau,\omega}}$-predictable localizing sequence for this process, and this sequence can be further modified
  into an $\F$-localizing sequence by an application of \cite[Theorem~IV.78, p.\,133]{DellacherieMeyer.78}. Thus, $M^{\tau,\omega}_{\tau+\cdot}$ is an $\F$-adapted $P^{\tau,\omega}$-$\F_{+}^{P^{\tau,\omega}}$-local martingale with a localizing sequence of $\F$-stopping times. By the tower property of the conditional expectation, this actually means that
  $M^{\tau,\omega}_{\tau+\cdot}$ is a $P^{\tau,\omega}$-$\F$-local martingale, with the same localizing sequence. As $\omega\in\Omega'$ was arbitrary, the proof is complete.
\end{proof}

For the rest of this section, we will be concerned with the process
\begin{equation*}
\widetilde{X}_t:= X_t -\sum_{0 \leq s \leq t} \big[\Delta X_s- h(\Delta X_s)\big], \quad t\geq 0;
\end{equation*}
recall that $h$ is a fixed truncation function. The process $\widetilde{X}$ has uniformly bounded jumps and differs from $X$ by a finite variation process; in particular, $X$ is a $P$-$\F$-semimartingale if and only if $\widetilde{X}$ is a $P$-$\F$-semimartingale, for any $P \in \fP(\Omega)$. In fact, if $P \in \fP_{sem}$, then as $\widetilde{X}$ has bounded jumps, it is a special semimartingale with a canonical decomposition $\widetilde{X}= M + B$; here $M$ is a right-continuous $P$-$\F$-local martingale and $B$ is an $\F$-predictable process with paths which are right-continuous and $P$-a.s.\ of finite variation.

\begin{proposition} \label{prop:candecinv}
  Let $\tau$ be a finite $\F$-stopping time, let $P \in \fP_{sem}$ and let $\widetilde{X}= M + B$ be the $P$-$\F$-canonical decomposition of $\widetilde{X}$. For $P$-a.e.\ $\omega \in \Omega$, we have $P^{\tau,\omega} \in \fP_{sem}$ and the canonical decomposition of $\widetilde{X}$ under $P^{\tau,\omega}$ is given by
  \begin{equation}\label{eq:canonDecomp}
  \widetilde{X}
   =  \big(M^{\tau,\omega}_{\tau +\cdot}-M_{\tau(\omega)}(\omega) \big)
   + \big(B^{\tau,\omega}_{\tau+ \cdot}-B_{\tau(\omega)}(\omega)\big).
  \end{equation}
\end{proposition}

\begin{proof}
  The right-continuous processes $B^{\tau,\omega}_{\tau+ \cdot}-B_{\tau(\omega)}(\omega)$ and $M^{\tau,\omega}_{\tau+ \cdot}-M_{\tau(\omega)}(\omega)$ are $\F$-adapted by Galmarino's test. As $\widetilde{X}$ is a $P$-$\F$-semimartingale with uniformly bounded jumps and $B$ has $P$-a.s.\ c\`adl\`ag paths, it follows that $M$ has $P$-a.s.\ c\`adl\`ag paths with uniformly bounded jumps; cf.\ \cite[Proposition~I.4.24, p.\,44]{JacodShiryaev.03}. Thus, we conclude from Lemma \ref{le:localmarttransinv} that for $P$-a.e.\ $\omega \in \Omega$,  $M^{\tau,\omega}_{\tau +\cdot}-M_{\tau(\omega)}(\omega)$ is a $P^{\tau,\omega}$-$\F$-local martingale. Moreover, we see that $B^{\tau,\omega}_{\tau+ \cdot}-B_{\tau(\omega)}(\omega)$ has $P^{\tau,\omega}$-a.s.\ finite variation paths for $P$-a.e.\ $\omega \in \Omega$; finally, it is adapted to the left-continuous filtration $\F_-=(\cF_{t-})_{t\geq0}$ and therefore $\F$-predictable as a consequence of \cite[Theorem~IV.97, p.\,147]{DellacherieMeyer.78}. We observe that~\eqref{eq:canonDecomp} holds identically, due to the definition of $\widetilde{X}$ and the fact that $X$ is the canonical process. As remarked above, this decomposition also implies that $X$ is a semimartingale under $P$.
\end{proof}

Next, we focus on the third characteristic. For ease of reference, we first state two simple lemmas.

\begin{lemma}\label{le:predcandinv}
  Let $W$ be a $\cP \otimes \B(\R^d)$-measurable function, let $\tau$ be a finite $\F$-stopping time and $\omega \in \Omega$. There exists a $\cP \otimes \B(\R^d)$-measurable function $\widetilde{W}$ such that
  \begin{equation}\label{eq:predcandinv}
    \widetilde{W}(\omega\otimes_\tau \tilde{\omega},\tau(\omega)+s,z)=W(\tilde{\omega},s,z), \ \ \ (\tilde{\omega},s,z) \in \Omega \times (0,\infty) \times \R^d.
  \end{equation}
  Moreover, if $W\geq 0$, one can choose $\widetilde{W}\geq 0$.
\end{lemma}

\begin{proof}
  Consider the function
  \begin{equation*}
    \widetilde{W}(\bar{\omega},s,z):= W\big(\bar\omega_{\cdot+\tau(\omega)}-\bar\omega_{\tau(\omega)},s-\tau(\omega),z\big)\,\1_{s>\tau(\omega)};
  \end{equation*}
  then~\eqref{eq:predcandinv} holds by definition. To show that $\widetilde{W}$ is $\cP \otimes \B(\R^d)$-measurable, we may first use the monotone class theorem to reduce to the case where $W$ is a product $W(\omega,t,x)=g(\omega,t)f(x)$. Using again the fact that a process is $\F$-predictable if and only if it is measurable and adapted to $\F_-$, cf.\ \cite[Theorem~IV.97, p.\,147]{DellacherieMeyer.78}, we then see that the predictability of $W$ implies the predictability of $\widetilde{W}$.
\end{proof}

\begin{lemma}\label{le:condexpinv}
   Let $P \in \fP(\Omega)$, let $\tau$ be a finite $\F$-stopping time and let $\nu$ be the $P$-$\F$-compensator of $\mu^X$. Then, for any $\cP \otimes \B(\R^d)$-measurable function $W\geq 0$, we have $P$-a.s.\ that
  \begin{equation*}
    E^P\Big[\int_\tau^\infty W(\cdot,s,z)\, \mu^X(\cdot,ds, dz) \, \Big| \, \cF_\tau\Big] =
    E^P\Big[\int_\tau^\infty W(\cdot,s,z)\, \nu(\cdot,ds, dz) \, \Big| \, \cF_\tau\Big].
  \end{equation*}
\end{lemma}

\begin{proof}
  By the definition of the compensator, we have
  \begin{equation*}
    E^P\Big[\int_0^\infty \overline{W}(\cdot,s,z)\, \mu^X(\cdot,ds, dz) \Big] =
    E^P\Big[\int_0^\infty \overline{W}(\cdot,s,z)\, \nu(\cdot,ds, dz) \Big]
  \end{equation*}
  for any $\cP \otimes \B(\R^d)$-measurable function $\overline{W}\geq 0$. Let $A \in \cF_\tau$ and define the $\F$-stopping time $\tau_A:=\tau\,\1_{A} + \infty \,\1_{A^c}$; then the claim follows by applying this equality to the function $\overline{W}:=W\,\1_{]\!]\tau_A,\infty[\![}$.
\end{proof}

\begin{proposition}\label{prop:compjumpinv}
  Let $P \in \fP(\Omega)$, let $\nu$ be the $P$-$\F$-compensator of $\mu^X$ and let $\tau$ be a finite $\F$-stopping time. Then, for $P$-a.e.\ $\omega \in \Omega$, the $P^{\tau,\omega}$-$\F$-compensator of $\mu^X$ is given by the random measure
  \begin{align*}
    D\mapsto \int_{\tau(\omega)}^{\infty} \int_{\R^d} \1_D(s-\tau(\omega),z) \, \nu(\omega\otimes_\tau\cdot,ds,dz),\quad \! D\in\cB(\R_+)\otimes \cB(\R^d).
  \end{align*}
\end{proposition}

\begin{proof}
  Denote by $\nu^{\omega}(\cdot,ds,dz)$ the above random measure. To see that it is $\F$-predictable, let $W$ be a $\cP \otimes \B(\R^d)$-measurable function. If $\widetilde{W}$ is as in Lemma~\ref{le:predcandinv}, then
  \begin{align*}
  & \ \int_0^t \int_{\R^d} W(\tomega,s,z)\,\nu^{\omega}(\tomega,ds,dz)\\
  = & \ \int_{\tau(\omega)}^{\tau(\omega)+t} \int_{\R^d} \widetilde{W}(\omega\otimes_\tau\tilde{\omega},s,z)\, \nu(\omega\otimes_\tau\tomega,ds,dz), \ \ \ \ (\tilde{\omega},t) \in \Omega \times (0,\infty)
  \end{align*}
  and the latter process is $\F$-predictable as a consequence of \cite[Theorem~IV.97, p.\,147]{DellacherieMeyer.78} and the fact that $\nu$ is an $\F$-predictable random measure. Thus, $\nu^{\omega}(\cdot,ds,dz)$ is $\F$-predictable for every $\omega\in\Omega$.

  Let $W\geq0$ be a $\cP \otimes \B(\R^d)$-measurable function and let $\widetilde{W}\geq 0$ be as in Lemma \ref{le:predcandinv}. Using the identity $X_\cdot=X^{\tau,\omega}_{\tau+\cdot}-X_\tau(\omega)$ and Lemma~\ref{le:condexpinv}, we obtain for $P$-a.e.\ $\omega \in \Omega$ that
  \begin{align*}
  & \ E^{P^{\tau,\omega}}\bigg[\int_0^\infty \int_{\R^d} W(\cdot,s,z)\,\mu^X(\cdot,ds,dz) \bigg]\\
  =& \ E^{P^{\tau,\omega}}\bigg[\int_0^\infty \int_{\R^d} \widetilde{W}(\omega\otimes_\tau\cdot,\tau(\omega) +s,z)\,\mu^{X^{\tau,\omega}_{\tau+\cdot}}(\cdot,ds,dz) \bigg]\\
  =& \ E^{P}\bigg[\int_\tau^\infty \int_{\R^d} \widetilde{W}(\cdot,s,z)\,\mu^{X}(\cdot,ds,dz) \, \bigg|\,\cF_\tau\bigg](\omega)\\
  =& \ E^{P}\bigg[\int_\tau^\infty \int_{\R^d} \widetilde{W}(\cdot,s,z)\,\nu(\cdot,ds,dz) \, \bigg|\,\cF_\tau\bigg](\omega)\\
  =& \ E^{P^{\tau,\omega}}\bigg[\int_{\tau(\omega)}^\infty \int_{\R^d} \widetilde{W}(\omega\otimes_\tau\cdot,s,z)\, ,\nu(\omega\otimes_\tau\cdot,ds,dz) \bigg]\\
  =& \ E^{P^{\tau,\omega}}\bigg[\int_0^\infty \int_{\R^d} W(\cdot,s,z)\,\nu^{\omega}(\cdot,ds,dz)\bigg].
  \end{align*}
  As $W\geq 0$ was arbitrary, it follows that $\nu^{\omega}(\cdot,ds,dz)$ is the $P^{\tau,\omega}$-$\F$-compensator of $\mu^X$ for $P$-a.e.\ $\omega \in \Omega$; cf.\ \cite[Theorem II.1.8, p.\,66]{JacodShiryaev.03}.
\end{proof}

We can now complete the proof of the main result of this section.

\begin{proof}[Proof of Theorem~\ref{th:charactinv}.]
  Let $P \in \fP_{sem}$ and let $\tau$ be a finite $\F$-stopping time. The formula $B^{\tau,\omega}_{\tau+ \cdot}-B_{\tau(\omega)}(\omega)$ for the first characteristic follows from Proposition~\ref{prop:candecinv} and the very definition of the first characteristic, whereas the formula for the third characteristic follows from Proposition~\ref{prop:compjumpinv} and the decomposition~\eqref{eq:nuDecomp}. Turning to the second characteristic, we recall from~\cite[Theorem~2.5]{NeufeldNutz.13a} that there exists an $\F$-predictable process $\hat{C}$ with the property that for any $P'\in\fP_{sem}$, $\hat{C}$ coincides $P'$-a.s.\ with the quadratic variation of the continuous local martingale part of $X$ under $P'$. The process $\hat{C}$ is constructed by subtracting the squared jumps of $X$ from the quadratic covariation $[X]$ which, in turn, is constructed in a purely pathwise fashion; moreover, the increments of $\hat{C}$ depend only on the increments of $X$. More precisely, it follows from the construction of $\hat{C}$ in the proof of \cite[Proposition~6.6(i)]{NeufeldNutz.13a} that for $P$-a.e.\ $\omega\in\Omega$,
  \[%
    \hat{C} = \hat{C}_{\tau + \cdot}^{\tau,\omega}-\hat{C}_{\tau(\omega)}(\omega)\quad P^{\tau,\omega}\as
  \]%
  Since the above holds in particular for $P'=P$ and $P'=P^{\tau,\omega}$, and since $\hat{C}$ is $\F$-predictable, this already yields the formula for the second characteristic under $P^{\tau,\omega}$. Finally, we observe that the assertion about $P\in\fP^{ac}_{sem}$ is a consequence of the general case that we have just established.
\end{proof}

\section{Products of Semimartingale Laws and~(A3)}\label{se:A3}

In this section, we first show that the product $\bar{P}$ of a semimartingale law $P\in\fP_{sem}$ and a $\fP_{sem}$-valued kernel $\kappa$ is again a semimartingale law. Then, we describe the associated characteristics and deduce the validity of Condition~(A3) for $\fP_\Theta$. While the naive way to proceed would be to construct directly the semimartingale decomposition under $\bar{P}$, some technical issues arise as soon as $\kappa$ has uncountably many values. For that reason, the first step will be achieved in a more abstract way using the Bichteler--Dellacherie criterion. Once the semimartingale property  for $\bar{P}$ is established, we know that the associated decomposition and characteristics exist and we can study them using the results of the previous section.

\begin{proposition}\label{prop:stabsem}
  Let $\tau$ be a finite $\F$-stopping time and let $P \in \fP_{sem}$. Moreover, let $\kappa: \Omega \to \fP(\Omega)$ be an $\cF_\tau$-measurable kernel with $\kappa(\omega) \in \fP_{sem}$ for $P$-a.e.\ $\omega \in \Omega$. Then, the  measure  $\bar{P}$ defined by
  \begin{equation*}
  \bar{P}(D):= \iint \1_D^{\tau,\omega}(\omega')\,\kappa(\omega, d\omega')\,P(d\omega), \ \ \ D \in \cF
  \end{equation*}
  is an element of $\fP_{sem}$.
\end{proposition}

\begin{proof}
  We recall that the $\bar{P}$-semimartingale property in $\F$ is equivalent to the one in the usual augmentation $\F^{\bar{P}}_{+}$; cf.\ \cite[Proposition~2.2]{NeufeldNutz.13a}. We shall use the Bichteler--Dellacherie criterion \cite[Theorem~VIII.80, p.\,387]{DellacherieMeyer.82} to establish the latter; namely, we show that
  if
  \begin{equation*}
    H^n=\sum_{i=1}^{k^n} h^n_i \, \1_{(t^n_{i-1},t^n_i]},\quad n\geq1
  \end{equation*}
  is a sequence of $\F^{\bar{P}}_{+}$-elementary processes such that $ H^n(t,\omega)\to0$ uniformly in $(t,\omega)$, then
  \[%
  \lim\limits_{n \to \infty} E^{\bar{P}}\bigg[\Big|\sum_{i=1}^{k^n}  h^n_i (X_{t^n_{i}\wedge t}-X_{t^n_{i-1} \wedge t})\Big| \wedge 1 \bigg] = 0,\quad t\geq0.
  \]%
  In fact, as $\bar{P}=P$ on $\cF_\tau$, it is clear that $X\,\1_{[\![0,\tau[\![}$ is a semimartingale and so it suffices to verify the above property for $X\,\1_{[\![\tau,\infty[\![}$ instead of $X$. To that end, by dominated convergence, it suffices to show that for $\bar{P}$-a.e.\ $\omega \in \Omega$,
  \[
   \lim_{n \to \infty} E^{\bar{P}}\bigg[\Big|\sum_{i=1}^{k^n}  h^n_i \big(X_{t^n_{i}\wedge t}\,\1_{\{\tau\leq t^n_{i} \wedge t\}} - X_{t^n_{i-1}\wedge t} \,\1_{\{\tau\leq t^n_{i-1} \wedge t\}} \big)\Big| \wedge 1\,\bigg| \, \cF_\tau \bigg](\omega) = 0,
  \]
  where $t\geq0$ is fixed. Define the $\cF_\tau$-measurable random variable $j^n$ by
  \[
    j^n:=\inf \{0\leq j\leq k^n|\, t_j^n\wedge t \geq \tau(\omega)\} \wedge k^n.
  \]
  Writing the above limit as a sum of two terms, it then suffices to show that for $\bar{P}$-a.e.\ $\omega \in \Omega$,
  \begin{align}
  \lim\limits_{n \to \infty} & \ E^{\bar{P}}\bigg[\Big|\sum_{i=j^n+1}^{k^n}  h^n_i \big(X_{t^n_{i}\wedge t} - X_{t^n_{i-1}\wedge t} \big )\Big| \wedge 1\,\bigg| \, \cF_\tau \bigg](\omega)  =0,   \label{eq:stabsem1}\\
  \lim\limits_{n \to \infty}  & \ E^{\bar{P}}\bigg[\Big|h^n_{j^n}\,X_{t^n_{j^n}\wedge t}\Big| \wedge 1\,\bigg| \, \cF_\tau \bigg](\omega) = 0 \label{eq:stabsem2}.
  \end{align}
  Indeed, as $h^n_{j^n}\to0$ uniformly, we have $|h^n_{j^n}\,X_{t^n_{j^n}\wedge t}|\to0$ $\bar{P}$-a.s.\ and hence~\eqref{eq:stabsem2} follows by dominated convergence.

  To show~\eqref{eq:stabsem1}, we may choose a $\cF_{t^n_{i-1}+}$-measurable version of each $h^n_i$. Then, as $\bar{P}^{\tau,\omega} = \kappa(\omega) \in \fP_{sem}$
  for $\bar{P}$-a.e.\ $\omega \in \Omega$ (cf.\ \cite[Lemma~2.7]{NutzVanHandel.12}), the reverse implication of the Bichteler--Dellacherie theorem applied to $\kappa(\omega)$ yields that
  \begin{align*}
  & \ \lim_{n \to \infty} E^{\bar{P}}\bigg[\Big|\sum_{i=j^n+1}^{k^n}  h^n_i \big(X_{t^n_{i}\wedge t} - X_{t^n_{i-1}\wedge t} \big )\Big| \wedge 1\,\bigg| \, \cF_\tau \bigg](\omega)\\
  = & \ \lim_{n \to \infty} E^{\bar{P}^{\tau,\omega}}\bigg[\Big|\sum_{i=j^n(\omega)+1}^{k^n}  (h_i^n)^{\tau, \omega} \big(X^{\tau,\omega}_{t^n_{i}\wedge t} - X^{\tau,\omega}_{t^n_{i-1}\wedge t} \big )\Big| \wedge 1\bigg] \\
  = & \ \lim_{n \to \infty} E^{\kappa(\omega)}\bigg[\Big|\sum_{i=j^n(\omega)+1}^{k^n}  (h_i^n)^{\tau, \omega} \big(X_{(t^n_{i}\wedge t)-\tau(\omega)} - X_{(t^n_{i-1}\wedge t)-\tau(\omega)} \big )\Big| \wedge 1\bigg] =0
  \end{align*}
  for $\bar{P}$-a.e.\ $\omega \in \Omega$, because $(H^n)^{\tau, \omega}$ defines a sequence of elementary processes converging uniformly to zero. This completes the proof.
\end{proof}

As announced, we can now proceed to establish~(A3) for $\fP_\Theta$.

\begin{proposition}\label{prop:A3}
Let $\Theta \subseteq \R^d \times \S^d_{+} \times \mathcal{L}$ be measurable and $P \in \fP_\Theta$. Moreover, let $\tau$ be a finite $\F$-stopping time and let $\kappa: \Omega \to \fP(\Omega)$ be an $\cF_\tau$-measurable kernel with $\kappa(\omega) \in \fP_\Theta$ for $P$-a.e.\ $\omega \in \Omega$. Then, the measure $\bar{P}$ defined by
\begin{equation*}
\bar{P}(D):= \iint \1_D^{\tau,\omega}(\omega')\,\kappa(\omega, d\omega')\,P(d\omega), \ \ \ D \in \cF
\end{equation*}
 is an element of $\fP_\Theta$.
\end{proposition}

\begin{proof}
   As a first step, we consider the special case $\Theta= \R^d \times \S^d_{+} \times \mathcal{L}$; then $\fP_\Theta$ is the entire set $\fP^{ac}_{sem}$.
   In view of Proposition \ref{prop:stabsem}, we already know that $\bar{P} \in \fP_{sem}$. Thus, the characteristics $(B,C,F(dz)\,dA)$ of $X$ under $\bar{P}$ and $\F$ are well defined; we show that they are absolutely continuous $\bar{P}$-a.s. As $B$ has paths of finite variation $\bar{P}$-a.s., we can write for $\bar{P}$-a.e.\ $\omega \in \Omega$ a decomposition
   \begin{equation*}
   B_t(\omega)=\int_0^t \varphi_s(\omega)\,ds + \psi_t(\omega),
   \end{equation*}
   where $\varphi, \psi$ are measurable functions and $\psi$ is $\bar{P}$-a.s.\ singular with respect to the Lebesgue measure.
   Since $\bar{P}=P$ on $\cF_{\tau}$ and $P \in \fP^{ac}_{sem}$, we have $dB\ll du$ on $[\![0,\tau]\!]$ $\bar{P}$-a.s. Therefore, it suffices to show that $dB\ll du$ on $[\![\tau,\infty[\![$ $\bar{P}$-a.s., or equivalently, that
   \begin{equation*}
     D:=\bigg\{ B_{\tau+\cdot} - B_\tau \neq \int_\tau^{\tau+\cdot} \varphi_s \,ds\bigg\}
   \end{equation*}
   is a $\bar{P}$-nullset.
   Indeed, it follows from Theorem~\ref{th:charactinv} that for $\bar{P}$-a.e.\ $\omega \in \Omega$, the first characteristic of $X$ under $\bar{P}^{\tau,\omega}$ is given by
  \begin{equation*}
  B^{\tau,\omega}_{\tau(\omega)+\cdot}-B_{\tau(\omega)}(\omega) = \int_{\tau(\omega)}^{\tau(\omega)+\cdot} \varphi^{\tau,\omega}_{s}  \, ds + \psi^{\tau,\omega}_{\tau+\cdot} - \psi_{\tau(\omega)}(\omega),
  \end{equation*}
  and $\psi^{\tau,\omega}_{\tau+\cdot}  - \psi_{\tau(\omega)}(\omega)$ is singular with respect to the Lebesgue measure.
  Moreover, for $\bar{P}$-a.e.\ $\omega \in \Omega$, we have $\bar{P}^{\tau,\omega}=\kappa(\omega) \in \fP^{ac}_{sem}$ and thus
  \begin{equation}\label{eq:stababs1}
   \kappa(\omega)\bigg\{B^{\tau,\omega}_{\tau+\cdot}-B_{\tau(\omega)}(\omega)\neq \int_{\tau(\omega)}^{\tau(\omega)+\cdot} \varphi^{\tau,\omega}_s \,ds\bigg\}
  =0.
  \end{equation}
  Define the set
  \begin{equation*}
  D^{\tau,\omega}:=\bigg\{B^{\tau,\omega}_{\tau+\cdot}-B_{\tau(\omega)}(\omega)\neq \int_{\tau(\omega)}^{\tau(\omega)+\cdot} \varphi^{\tau,\omega}_s \,ds\bigg\},\quad \omega \in \Omega;
  \end{equation*}
  then~\eqref{eq:stababs1} states that
  \begin{equation*}
    \kappa(\omega)\big(D^{\tau,\omega}\big) =0 \ \ \mbox{for} \ \bar{P}\mbox{-a.e.} \ \omega \in \Omega.
  \end{equation*}
  As $\kappa$ is $\cF_\tau$-measurable and $\bar{P}=P$ on $\cF_\tau$, this equality holds also for
  $P$-a.e.\ $\omega \in \Omega$. Using Fubini's theorem and the fact that $\1_{D}^{\tau,\omega}=\1_{D^{\tau,\omega}}$, we conclude that
  \begin{align*}
  \bar{P}(D)&=\int_\Omega \int_\Omega \1_{D}^{\tau,\omega}(\omega')\, \kappa(\omega,d\omega')\,P(d\omega)
  = \int_\Omega \kappa(\omega)\big(D^{\tau,\omega}\big)\,P(d\omega) = 0
  \end{align*}
  as claimed. The proof of absolute continuity for the processes $C$ and $A$ is similar; we use the corresponding formulas from Theorem~\ref{th:charactinv}. This completes the proof for the special case $\Theta= \R^d \times \S^d_{+} \times \mathcal{L}$.

  Next, we consider the case of a general subset $\Theta\subseteq \R^d \times \S^d_{+} \times \mathcal{L}$. By the above, $\bar{P} \in \fP^{ac}_{sem}$; we write $(\int b_s \,ds, \int c_s \, ds, F_{s}\,ds)$ for the characteristics of $X$ under $\bar{P}$.
  Since $\bar{P}=P$ on $\cF_{\tau}$ and $P \in \fP_\Theta$, we have $(b,c,F) \in \Theta$ on $[\![0,\tau]\!]$, $du\times \bar{P}$-a.s., and it suffices to show that $(b,c,F) \in \Theta$ on $[\![\tau,\infty[\![$, $du\times \bar{P}$-a.s. That is, we need to show that
  \begin{equation*}
    R:=\Big\{(u,\omega) \in [\![\tau(\omega),\infty[\![ \, \Big| \, \big(b_u(\omega), c_u(\omega), F_{\omega,u}\big) \notin \Theta \Big\}
  \end{equation*}
  is a $du \times \bar{P}$-nullset. By Theorem~\ref{th:charactinv}, $\bar{P}^{\tau,\omega} \in \fP^{ac}_{sem}$ for $\bar{P}$-a.e.\ $\omega \in \Omega$  and the differential characteristics of $X$ under $\bar{P}^{\tau,\omega}$ are
  \begin{equation*}%
   \big(b^{\tau,\omega}_{\tau+\cdot},\, c^{\tau,\omega}_{\tau+\cdot},\, F^{\tau,\omega}_{\tau+\cdot}\big).
  \end{equation*}
  Similarly as in~\eqref{eq:stababs1}, this formula and the fact that $\bar{P}^{\tau,\omega}=\kappa(\omega)\in \fP_ {\Theta}$ for $\bar{P}$-a.e.\ $\omega \in \Omega$ imply that
  \begin{equation*}%
   \big(du\times\kappa(\omega)\big)\Big\{(u,\omega') \in [\![0,\infty[\![ \, \Big| \, \big(b^{\tau,\omega}_{\tau+u}(\omega'),  c^{\tau,\omega}_{\tau+u}(\omega'),F^{\tau,\omega}_{\omega',\tau+u}\big)  \notin \Theta \Big\}
  =0
  \end{equation*}
  for $\bar{P}$-a.e.\ $\omega \in \Omega$.
  If we define
  \[
    R^{\tau,\omega}:= \Big\{(u,\omega') \in [\![\tau(\omega),\infty[\![ \, \Big| \, \big(b^{\tau,\omega}_{u}(\omega'),  c^{\tau,\omega}_{u}(\omega'),F^{\tau,\omega}_{\omega',u}\big)  \notin \Theta \Big\},
  \]
  then this implies that
  \begin{equation*}
    \big(du\times\kappa(\omega)\big)\big(R^{\tau,\omega}\big) =0 \ \ \mbox{for} \ \bar{P}\mbox{-a.e.} \ \omega \in \Omega.
  \end{equation*}
  Again, this holds also for $P$-a.e.\ $\omega \in \Omega$ and as $\1_{R}^{\tau,\omega}=\1_{R^{\tau,\omega}}$, Fubini's theorem yields that
  \begin{align*}
    (du\times \bar{P})(R)&=\int_\Omega \int_\Omega \int_0^\infty \1_{R}^{\tau,\omega}(u,\omega')\,du\, \kappa(\omega,d\omega')\,P(d\omega)\\
    &= \int_\Omega \big(du\times\kappa(\omega)\big)(R^{\tau,\omega})\,P(d\omega) = 0.
  \end{align*}
  This completes the proof.
\end{proof}

\section{Connection to PIDE}\label{se:PIDE}

In this section, we relate the nonlinear L\'evy process to a PIDE. Throughout, we fix a measurable set $\Theta \subseteq \R^d \times \S_{+}^d \times \mathcal{L}$ satisfying the conditions~\eqref{eq:intcond} and~\eqref{eq:limitcond} which, for convenience, we state again as
\begin{align}
\mathcal{K}:=\sup_{(b,c,F) \in \Theta} \Big\{\int_{\R^d} |z|\wedge |z|^2 \, F(dz) + |b| + |c| \Big\}<\infty, \label{eq:intcondRestated}
\\
\lim_{\varepsilon\to 0} \cK_\eps=0\quad\mbox{for}\quad \cK_\eps :=\sup_{F \in \Theta_3} \int_{|z|\leq \varepsilon} |z|^2 \, F(dz), \label{eq:limitcondRestated}
\end{align}
where $\Theta_3=\proj_3 \Theta$ is the canonical projection of $\Theta$ onto $\cL$.
Our aim is to show that for given boundary condition $\psi \in C_{b,Lip}(\R^d)$, the value function
\[%
  v(t,x):=\cE(\psi(x+X_{t})\big) \equiv \sup_{P \in \fP_\Theta} E^P\big[\psi(x+X_{t})\big], \quad (t,x) \in [0,\infty)\times \R^d
\]%
is the unique viscosity solution of the PIDE~\eqref{eq:PIDE}.

The existence part relies on the following dynamic programming principle for $v$; it is essentially a special case of the semigroup property stated in Theorem~\ref{th:mainLevy}(ii).

\begin{lemma}\label{le:DPP}
  For all $0\leq u\leq t <\infty$ and $x\in\R^d$, we have
  \begin{equation*}
   v(t,x)= \cE\big( v(t-u,x+X_{u})\big).
  \end{equation*}
\end{lemma}

\begin{proof}
  Let $0\leq u \leq t<\infty$. As $X$ is the canonical process, we have that
  \[
    \cE_u\big(\psi(x+X_t)\big)(\omega)=\cE\big(\psi(x+X_u(\omega)+X_{t-u})\big)=v(t-u,x+X_{u}(\omega)),\quad \omega\in\Omega.
  \]
  Applying $\cE(\cdot)$ on both sides, Theorem~\ref{th:mainLevy}(ii) yields that
  \[
    v(t,x) = \cE\big(\cE_u\big(\psi(x+X_t)\big)\big) = \cE\big(v(t-u,x+X_{u})\big)
  \]
  as claimed.
\end{proof}

During most of this section, we will be concerned with a fixed law $P \in \fP^{ac}_{sem}$ and we may use the usual augmentation $\F^P_{+}$ to avoid any subtleties related to stochastic analysis. This is possible because, as mentioned in Section~\ref{se:mainResults}, the characteristics associated with $\F$ and $\F^P_{+}$ coincide $P$-a.s.
To fix some notation, recall that under $P \in \fP^{ac}_{sem}$, the process $X$ has the canonical representation
\begin{align}\label{eq:canrep}
 X_t
  & = \int_0^t b^P_s \, ds + X^{c,P}_t + X^{d,P}_t + \int_0^t \int_{\R^d} \big[z-h(z) \big] \, \mu^X(ds,dz),
\end{align}
where $X^{c,P}$ is the continuous local martingale part of $X$ with respect to $P$-$\F^P_{+}$, $F^P_{s}(dz) \, ds$ is  the compensator of $\mu^X(ds,dz)$ and
\[
  X^{d,P}_t:= \int_0^t \int_{\R^d} h(z) \, \big(\mu^X(ds,dz) - F^P_{s}(dz)\,ds\big)
\]
is a purely discontinuous $P$-$\F^P_{+}$-local martingale;  cf.\ \cite[Theorem~2.34, p.\,84]{JacodShiryaev.03}. In the subsequent proofs, $C$ is a constant whose value may change from line to line.

The following simple estimate will be used repeatedly.

\begin{lemma}\label{le:special}
  There exists a constant $C_\cK$ such that
  \begin{equation}\label{eq:Wachstum}
   E^P\bigg[\sup_{0\leq u \leq t} |X_u| \bigg]\leq C_{\cK}\, (t+ t^{1/2}),\quad t\geq0 \quad \mbox{for all} \quad P \in \fP_\Theta.
  \end{equation}
\end{lemma}
\begin{proof}
  Let $P \in \fP_{\Theta}$; then Jensen's inequality and~\eqref{eq:intcondRestated} imply that
  \begin{align*}
  E^P\Big[ \big|[X^{d,P}]_t \big|^{1/2}\Big]
  & \leq E^P\bigg[\int_0^t \int_{\R^d} |h(z)|^2  \, \mu^X(ds,dz) \bigg]^{1/2} \\
  & \leq C\, E^P\bigg[\int_0^t \int_{\R^d} |z|^2\wedge 1  \, F_{s}(dz)\,ds \bigg]^{1/2} \\
  & \leq  C\, \mathcal{K}^{1/2} t^{1/2}
  \end{align*}
  and so the Burkholder--Davis--Gundy (BDG) inequalities yield that
  \begin{equation}\label{eq:estimateJumpMartBDG}
  E^P\bigg[\sup_{0\leq u \leq t} \big| X^{d,P}_u \big| \bigg] \leq C\, E^P\Big[\big|[X^{d,P}]_t \big|^{1/2}\Big] \leq C_\cK \,t^{1/2}.
  \end{equation}
  Similarly, \eqref{eq:intcondRestated} also implies that
  \begin{equation}\label{eq:estimateContMartBDG}
    E^P\bigg[\sup_{0\leq u \leq t} \big| X^{c,P}_u \big| \bigg]\leq C_\cK\, t^{1/2},\quad E^P\bigg[\sup_{0\leq u \leq t} \bigg| \int_0^u b_s \,ds \bigg| \bigg] \leq C_\cK\, t
  \end{equation}
  and
  \[
    E^P\bigg[\sup_{0\leq u \leq t}  \Big| \int_0^u \int_{\R^d} \big[z-h(z)\big] \, \mu^X(ds,dz) \Big| \bigg] \leq  C_\cK\, t.
  \]
  The result now follows from the decomposition~\eqref{eq:canrep}.
\end{proof}

We deduce the following regularity properties of $v$.

\begin{lemma}\label{le:continuous}
  The value function $v$ is uniformly bounded by $\|\psi\|_\infty$ and jointly continuous. More precisely, $v(t,\cdot)$ is Lipschitz continuous with constant $\Lip(\psi)$ and $v(\cdot,x)$ is locally $1/2$-H\"older continuous with a constant depending only on $\Lip(\psi)$ and $\cK$.
\end{lemma}

\begin{proof}
  The boundedness and the Lipschitz property follow directly from the definition of $v$. Let $0\leq u\leq t$, then Lemma~\ref{le:DPP}, the Lipschitz continuity of $v(t,\cdot)$ and the estimate~\eqref{eq:Wachstum} show that
  \begin{align*}
    \big|v(t,x)-v(t-u,x)\big|&=\big| \cE\big(v(t-u,x+X_u)-v(t-u,x)\big)\big|\\
    &\leq C \, \cE\big(|X_u|\big) \\
    &\leq C  \, (u + u^{1/2} ).
  \end{align*}
  The H\"older continuity from the right is obtained analogously.
\end{proof}

\subsection{Existence}\label{se:exist}

Consider the PIDE introduced in \eqref{eq:PIDE}; namely,
\begin{equation*}\label{eq:PIDE2}
\partial_t v(t,x)-G\big(D_x v(t,x), D^2_{xx} v(t,x), v(t,x+\cdot)\big)=0, \quad v(0,x)=\psi(x) %
\end{equation*}
for $(t,x) \in (0,\infty) \times \R^d$, where the nonlinearity $G(p,q,f(\cdot))$ is given by
\[
\sup_{(b,c,F) \in \Theta}\bigg\{
p b + \frac{1}{2} \tr[q c] + \int_{\R^d} \big[f(z)-f(0)- D_x f(0) h(z)\big] F(dz)
     \bigg\}. \label{eq:PIDE-G2}
\]
We recall that $\psi \in C_{b,Lip}(\R^d)$ and $v(t,x)=\cE(\psi(x+X_t))$.

\begin{proposition}\label{prop:PIDEex}
  The value function $v$ of~\eqref{eq:val} is a viscosity solution of the PIDE~\eqref{eq:PIDE}.
\end{proposition}

\begin{proof}
  The basic line of argument is standard in stochastic control. We detail the proof because the presence of small jumps necessitates additional arguments; this is where the condition~\eqref{eq:limitcondRestated} comes into play.

  By Lemma \ref{le:continuous}, $v$ is continuous on $[0,\infty)\times \R^d$, and we have $v(0,\cdot)= \psi$ by the definition of $v$. We show that $v$ is a viscosity subsolution of~\eqref{eq:PIDE}; the supersolution property is proved similarly.

  Let $(t,x) \in (0,\infty)\times \R^d$ and let $\varphi \in C^{2,3}_b((0,\infty)\times \R^d)$ be such that $\varphi\geq v$ and $\varphi(t,x)=v(t,x)$.
  For $0<u<t$, Lemma~\ref{le:DPP} shows that
  \begin{equation}\label{eq:subsol1}
    0= \sup_{P \in \fP_\Theta} E^P\big[v(t-u,x+X_u)-v(t,x)\big]  \leq \sup_{P \in \fP_\Theta} E^P\big[\varphi(t-u,x+X_u)-\varphi(t,x)\big].
  \end{equation}
  We fix $P \in \fP_\Theta$ and recall that $(b^P,c^P,F^P)$ are the differential characteristics of $X$ under $P$. Applying It\^o's formula, we obtain that $P$-a.s.,
  \begin{align}
    & \ \varphi(t-u,x+X_u)-\varphi(t,x) \nonumber\\
    =&  \  \int_0^u \int_{\R^d} D_x \varphi(t-s,x+X_{s-}) \, d(X^{c,P}_s+X^{d,P}_s) \nonumber\\
    &   +\int_0^u -\partial_t \varphi(t-s,x+X_{s-})\, ds + \int_0^u D_x \varphi(t-s,x+X_{s-}) b^P_s\,ds \nonumber\\
    &  + \frac{1}{2}\int_0^u \tr\big[D^2_{xx} \varphi(t-s,x+X_{s-}) \, c^P_s \big]\,ds  \nonumber\\
    &  + \int_0^u \int_{\R^d} \Big[\varphi(t-s,x+X_{s-}+z)- \varphi(t-s,x+X_{s-}) \nonumber\\
    & \phantom{  + \int_0^u \int_{\R^d} \Big[\varphi(t-s,x+X_{s-}+z)}- D_x \varphi(t-s,x+X_{s-}) h(z)\Big]\mu^X(ds,dz). \label{eq:subsol2}
  \end{align}
  Since $\varphi \in C^{2,3}_b$, it follows from~\eqref{eq:estimateJumpMartBDG} and~\eqref{eq:estimateContMartBDG} that the first integral in~\eqref{eq:subsol2} is a true martingale; in particular,
  \begin{equation}\label{eq:subsol3.0}
    E^P\bigg[\int_0^u \int_{\R^d} D_x \varphi(t-s,x+X_{s-}) \, d(X^{c,P}_s+X^{d,P}_s)\bigg]=0,\quad u\geq0.
  \end{equation}
  Using~\eqref{eq:intcondRestated} and~\eqref{eq:Wachstum}, we can estimate the expectations of the other terms in~\eqref{eq:subsol2}. Namely, we have
  \begin{align}
    & \ E^P\bigg[\int_0^u D_x \varphi(t-s,x+X_{s-}) b^P_s\, ds\bigg] \nonumber\\
   \leq & \ \int_0^u E^P\Big[ \big|D_x \varphi(t-s,x+X_{s-})-D_x \varphi(t,x)\big|\, |b^P_s| + D_x \varphi(t,x) b^P_s \Big]\, ds \nonumber\\
   \leq & \ \int_0^u E^P\Big[C\,\big(s + |X_{s-}|\big) \Big] + E^P\Big[ D_x \varphi(t,x) b^P_s\Big]\, ds \nonumber \\
   \leq & \  C \, (u^2 + u^{3/2})+ \int_0^u E^P\Big[ D_x \varphi(t,x) b^P_s\Big]\, ds,  \label{eq:Lipb}
   \end{align}
   and similarly
   \begin{equation}\label{eq:Lipt}
      E^P\bigg[ \int_0^u -\partial_t \varphi(t-s,x+X_{s-})\, ds\bigg] \leq \int_0^u -\partial_{t} \varphi(t,x)\, ds + C \, (u^2 + u^{3/2})
   \end{equation}
   as well as
   \begin{align}
   & \ E^P\bigg[\int_0^u \tr\big[D^2_{xx} \varphi(t-s,x+X_{s-}) c^P_s\big]\, ds\bigg] \nonumber\\
   \leq & \
   \int_0^u E^P\Big[\tr\big[ D^2_{xx} \varphi(t,x)\,c^P_s\big]\Big]\, ds + C \, (u^2 + u^{3/2}). \label{eq:Lipc}
   \end{align}
   For the last term in~\eqref{eq:subsol2}, we shall distinguish between jumps smaller and larger than a given $\eps>0$, where $\eps$ is such that $h(z)=z$ on $\{|z| \leq \varepsilon\}$. Indeed, a Taylor expansion shows that there exist $\xi_{z} \in \R^d$ such that $P$-a.s., the integral can be written as the sum
  \begin{align}
      & \int_0^u \int_{|z|>\varepsilon} \Big[\varphi(t-s,x+X_{s-}+z)- \varphi(t-s,x+X_{s-})  \nonumber\\
      & \phantom{\int_0^u \int_{\R^d} \Big[\varphi(t-s,x+X_{s-}+z)}- D_x \varphi(t-s,x+X_{s-}) h(z)\Big]\mu^X(ds,dz)  \nonumber\\
      & + \int_0^u \int_{|z|\leq \varepsilon} \frac{1}{2}\,\tr \big[D^2_{xx}\varphi(t-s,x+X_{s-}+ \xi_z) \,z z^\top \big]  \mu^X(ds,dz). \label{eq:subsol3.1}
\end{align}
By~\eqref{eq:intcondRestated}, both of these expressions are $P$-integrable. Using the same arguments as in \eqref{eq:Lipb}, the first integral satisfies
\begin{align}
& \ E^P\bigg[\int_0^u \int_{|z|>\varepsilon} \big[\varphi(t-s,x+X_{s-}+z)-\varphi(t-s,x+X_{s-})\nonumber\\
& \phantom{ \ E^P\bigg[\int_0^u \int_{|z|>\varepsilon} \big[\varphi(t-s,x+X_{s-}}-D_x\varphi(t-s,x+X_{s-}) h(z)\big]\,F_{s}(dz)\,ds\bigg] \nonumber\\
 \leq & \ E^P\bigg[\int_0^u\int_{|z|>\varepsilon} \big[\varphi(t,x+z)-\varphi(t,x)-D_x\varphi(t,x) h(z)\big]
 \,F^P_{s}(dz) \,ds \bigg] \nonumber\\
 & \ + C\, C_\eps \,(u^2 + u^{3/2}), \label{eq:Lipv}
\end{align}
where
\[
  C_\eps := \sup_{F \in \Theta_3} \int_{|z|>\varepsilon} 1 \, F(dz)
\]
is finite for every fixed $\eps>0$ due to~\eqref{eq:intcondRestated}. For the second integral in~\eqref{eq:subsol3.1}, we have
\begin{align}
& \ E^P\bigg[\int_0^u \int_{|z|\leq \varepsilon} \frac{1}{2}\,\tr \big[D^2_{xx}\varphi(t-s,x+X_{s-}+ \xi_z) \,z z^\top \big]  \mu^X(ds,dz)\bigg] \nonumber \\
= & \ E^P\bigg[\int_0^u \int_{|z|\leq \varepsilon} \frac{1}{2}\,\tr \big[D^2_{xx}\varphi(t-s,x+X_{s-}+ \xi_z) \,z z^\top \big]  F^P_{s}(dz) \, ds \bigg] \nonumber\\
\leq & \ C\, \cK_\eps\,u; \label{eq:subsol3.2}
\end{align}
recall~\eqref{eq:limitcondRestated}. Thus, taking expectations in \eqref{eq:subsol2} and using \eqref{eq:subsol3.0}--\eqref{eq:subsol3.2}, we obtain for small $\eps>0$ that
 \begin{align}
 & \ E^P\Big[\varphi(t-u,x+X_u)-\varphi(t,x)\Big] \nonumber\\
 \leq & \ \int_0^u E^P\bigg[-\partial_t \varphi(t,x) + D_x\varphi(t,x)  b^P_s + \frac{1}{2} \tr\big[D^2_{xx} \varphi(t,x) \, c^P_s\big] \nonumber \\
 & \ \phantom{\int_0^u E^P\bigg[} +  \int_{|z|>\varepsilon} \big[\varphi(t,x+z)-\varphi(t,x)- D_x \varphi(t,x) h(z)\big] \, F^P_{s}(dz)\bigg]\,ds \nonumber\\
 &   \ + C \,\cK_\eps\,u + C\, C_{\varepsilon}\,(u^2 + u^{3/2}) \nonumber\\
\leq & \ -u\partial_t \varphi(t,x)  + u \sup_{(b,c,F) \in \Theta} \bigg\{ D_x \varphi(t,x) b + \frac{1}{2} \tr\big[D^2_{xx} \varphi(t,x) \, c\big] \nonumber \\
 &   \ \quad\quad + \int_{|z|>\varepsilon} \big[\varphi(t,x+z)-\varphi(t,x)- D_x \varphi(t,x) h(z)\big] \, F(dz) \bigg\} \nonumber\\
 &   \ + C \,\cK_\eps\,u + C\, C_{\varepsilon}\,(u^2 + u^{3/2}). \label{eq:subsol5}
 \end{align}
 Regarding the integral in this expression, we note that for each $F \in \Theta_3$,
 \begin{align}
  & \ \int_{|z|>\varepsilon} \big[\varphi(t,x+z)-\varphi(t,x)- D_x \varphi(t,x) h(z)\big] \, F(dz) \nonumber \\
  \leq & \ \int_{\R^d} \big[\varphi(t,x+z)-\varphi(t,x)- D_x \varphi(t,x) h(z)\big] \, F(dz) \nonumber\\
    & \ + \bigg| \int_{|z|\leq \varepsilon} \big[\varphi(t,x+z)-\varphi(t,x)- D_x \varphi(t,x) h(z)\big] \, F(dz) \bigg| \nonumber \\
    \leq & \ \int_{\R^d} \big[\varphi(t,x+z)-\varphi(t,x)- D_x \varphi(t,x) h(z)\big] \, F(dz)+ C \, \cK_\eps
   \label{eq:subsol6}
 \end{align}
 by a Taylor expansion as above.
   We deduce from \eqref{eq:subsol5}, \eqref{eq:subsol6} and the definition of $G$ that
   \begin{align*}
    & \ E^P\big[\varphi(t-u,x+X_u)-\varphi(t,x)\big] \nonumber\\
      \leq & \ -u\partial_t \varphi(t,x)   + u \,G\big(D_x \varphi(t,x), D^2_{xx} \varphi(t,x), \varphi(t,x+\cdot)\big)  \nonumber \\
       & \ + C \,\cK_\eps\,u + C\, C_{\varepsilon}\,(u^2 + u^{3/2}). %
   \end{align*}
   By~\eqref{eq:subsol1}, it follows that
   \begin{align*}
     0  \leq \ &-u\partial_t \varphi(t,x)  + u G\big(D_x \varphi(t,x), D^2_{xx} \varphi(t,x), \varphi(t,x+\cdot)\big) \\
     & \ + C \,\cK_\eps\,u + C\, C_{\varepsilon}\,(u^2 + u^{3/2}).
   \end{align*}
   Now divide by $u$ and let first $u$ and then $\eps$ tend to zero. As $\cK_\eps\to0$ by~\eqref{eq:limitcondRestated}, we obtain that
   \[
     0  \leq -\partial_t \varphi(t,x)  + G\big(D_x \varphi(t,x), D^2_{xx} \varphi(t,x), \varphi(t,x+\cdot)\big)
   \]
   as desired.
\end{proof}
\subsection{Uniqueness}\label{se:unique}

The aim of this subsection is to show that a comparison principle holds for the PIDE~\eqref{eq:PIDE}; in particular, this will establish the uniqueness of the solution. We denote by $\USC_b((0,\infty)\times \R^d)$ the set of all bounded upper  semicontinuous functions on $(0,\infty)\times \R^d$. Similarly, $\LSC_b$ stands for the bounded lower semicontinuous functions, and $\SC_b:=\USC_b\cup\LSC_b$.

\begin{proposition}\label{pr:comparisonPIDE}
Let $u \in \USC_b([0,\infty) \times \R^d)$ be a viscosity subsolution and let $v \in \LSC_b([0,\infty) \times \R^d)$ be a viscosity supersolution of~\eqref{eq:PIDE}. If $u(0,\cdot), v(0,\cdot) \in C_{b,Lip}(\R^d)$ and $u(0,\cdot)\leq v(0,\cdot)$, then $u\leq v$.
\end{proposition}

The proof proceeds through the following general result, essentially due to~\cite{HuPeng.09levy} (which, in turn, draws from~\cite{BarlesImbert.08, JakobsenKarlsen.06}).

\begin{lemma}\label{le:genComparison}
  Let $G:\R^d \times \S^d \times C^2_b(\R^d) \to \R$ and suppose there exist functions $G^\kappa: \R^d \times \S^d \times \SC_b(\R^d) \times C^2(\R^d) \to \R$, $\kappa\in (0,1)$ such that Conditions~(C1)--(C9) below are satisfied. Then the assertion of Proposition~\ref{pr:comparisonPIDE} holds for
 \begin{equation*}%
 \partial_t v(t,x)-G\big(D_x v(t,x), D^2_{xx} v(t,x), v(t,x+\cdot)\big)=0,\quad (t,x)\in [0,\infty) \times \R^d.
 \end{equation*}
\end{lemma}

\begin{proof}
  This is essentially the result of \cite[Corollary~53]{HuPeng.09levy}. The only difference is that our Condition~(C8) below is slightly weaker than its analogue \cite[Theorem~51, Condition~(i)]{HuPeng.09levy}. An inspection of the proof of \cite[Theorem~51]{HuPeng.09levy} shows that the result remains true under the weaker condition.
\end{proof}

The conditions mentioned in the preceding lemma run as follows.

\begin{description}
  \item[(C1)] Let $(t_k,x_k,p_k,q_k) \to (t,x,p,q)$ in $(0,\infty)\times \R^d\times \R^d\times\S^d$.
  Moreover, let $f_k, f\in C^{1,2}_b((0,\infty)\times \R^d)$ be such that $f_k(t_k,x_k+\cdot) \to f(t,x+\cdot)$ locally uniformly on $\R^d$, $D_x f_k \to D_x f$ and $D^2_{xx} f_k \to D^2_{xx} f$ locally uniformly on $(0,\infty)\times \R^d$, and $(f_k)_{k \in \N}$ is uniformly bounded.
  Then
  \begin{equation*}
  G(p_k,q_k,f_k(t_k,x_k+\cdot)) \to G(p,q,f(t,x+\cdot)).
  \end{equation*}

  \item[(C2)] Let $(t,x,p,q_1,q_2)\in (0,\infty)\times \R^d\times \R^d\times\S^d\times\S^d$ be such that $q_1 \geq q_2$ and
  let $f_1, f_2 \in C^{1,2}_b((0,\infty)\times \R^d)$ be such that $(f_1-f_2)(t,\cdot)$ has a global minimum at $x$. Then
  \begin{equation*}
  G(p,q_1,f_1(t,x+\cdot)) \geq G(p,q_2,f_2(t,x+\cdot)).
  \end{equation*}

  \item[(C3)] Let $(t,x,p,q)\in (0,\infty)\times \R^d\times \R^d\times\S^d$ and $f \in C^{1,2}_b((0,\infty)\times \R^d)$. Then
  \begin{equation*}
  G(p,q,f(t,x+\cdot)+c) = G(p,q,f(t,x+\cdot)),\quad c \in \R.
  \end{equation*}

  \item[(C4)] Let $(t,x,p,q)\in (0,\infty)\times \R^d\times \R^d\times\S^d$ and
  let $f \in C^{1,2}_b((0,\infty)\times \R^d)$. Then
  \begin{equation*}
  G^\kappa(p,q, f(t,x+\cdot), f(t,x+\cdot)) = G(p,q, f(t,x+\cdot)),\quad \kappa \in (0,1).
  \end{equation*}

  \item[(C5)] Let $(t,x,p,q_1,q_2)\in (0,\infty)\times \R^d\times \R^d\times\S^d\times\S^d$ be such that $q_1 \geq q_2$,
  let $f_1   \in \LSC_b((0,\infty)\times \R^d)$ and $f_2 \in \USC_b((0,\infty)\times \R^d)$ be such that $(f_1-f_2)(t,\cdot)$ has a global minimum at $x$ and
  let $g_1, g_2 \in   C^{1,2}_b((0,\infty)\times \R^d)$ be such that $(g_1-g_2)(t,\cdot)$ has a global minimum at $x$. Then, for all $\kappa\in (0,1)$,
  \begin{equation*}
  G^\kappa(p,q_1, f_1(t,x+\cdot), g_1(t,x+\cdot)) \geq G^\kappa(p,q_2, f_2(t,x+\cdot), g_2(t,x+\cdot)).
  \end{equation*}

  \item[(C6)] Let $(t,x,p,q)\in (0,\infty)\times \R^d\times \R^d\times\S^d$, $f\in \SC_b((0,\infty)\times \R^d)$ and $g \in C^{1,2}_b((0,\infty)\times \R^d)$.
  Then, for all $\kappa \in (0,1)$ and $c_1, c_2 \in \R$,
  \begin{equation*}
  G^\kappa(p,q, f(t,x+\cdot)+c_1, g(t,x+\cdot)+c_2) = G^\kappa(p,q, f(t,x+\cdot), g(t,x+\cdot)).
  \end{equation*}

  \item[(C7)]  Let $(t,x,p,q)\in (0,\infty)\times \R^d\times \R^d\times\S^d$,
  let $f   \in \SC_b((0,\infty)\times \R^d)$
  and let $f_n,g \in   C^{1,2}_b((0,\infty)\times \R^d)$ be such that $f_n(t,\cdot)\to f(t,\cdot)$ locally uniformly on $\R^d$ and $(f_n)_{n \in \N}$ is uniformly bounded.
  Then, for all $\kappa \in (0,1)$,
  \begin{equation*}
  G^\kappa(p,q, f_n(t,x+\cdot), g(t,x+\cdot)) \to G^\kappa(p,q, f(t,x+\cdot), g(t,x+\cdot)).
  \end{equation*}

  \item[(C8)] There exists a constant $C>0$ such that
  \begin{align*}
  & \ |G^\kappa(p_1,q_1,f(t,\cdot)+ \psi(\cdot), g(t,\cdot)+ \psi(\cdot))
  - G^\kappa(p_2,q_2,f(t,\cdot), g(t,\cdot))| \\
  \leq & \ C\big(|p_1-p_2|+ |q_1-q_2|+ \|D_x \psi\|_\infty +\|D^2_{xx}\psi\|_\infty)\big)
  \end{align*}
  for all $\kappa \in (0,1)$, $t \in (0,\infty)$, $p_1, p_2 \in \R^d$, $q_1, q_2 \in \S^d$, $f \in \SC_b((0,\infty) \times \R^d)$, $g \in C^{1,2}((0,\infty) \times \R^d)$ and $\psi \in C^2_b(\R^d)$.

  \item[(C9)] Let $(t,x,p,q)\in (0,\infty)\times \R^d\times \R^d\times\S^d$,
  let $f \in \SC_b((0,\infty)\times \R^d)$ and
  let $g_1, g_2 \in C^{1,2}((0,\infty)\times \R^d)$ satisfy $D_x g_1(t,x)= D_x g_2(t,x)$.  Then
  \begin{equation*}
  \lim\limits_{\kappa \to 0} |G^\kappa(p,q, f(t,x+ \cdot), g_1(t,x+\cdot))-
  G^\kappa(p,q, f(t,x+ \cdot), g_2(t,x+\cdot))| =0.
  \end{equation*}
\end{description}

In order to deduce Proposition~\ref{pr:comparisonPIDE} from Lemma~\ref{le:genComparison}, we define the auxiliary functions $G^\kappa:\R^d \times \S^d \times \SC_b(\R^d) \times C^{2}( \R^d) \to \R$, $\kappa\in(0,1)$ by
\begin{multline}
G^\kappa(p,q,f(\cdot),g(\cdot)) \ :=  \ \sup_{(b,c,F)\in \Theta} \bigg\{ \int_{|z|>\kappa} [f(z)-f(0)  - D_x g(0) h(z)] \, F(dz)\\
 + \int_{|z|\leq \kappa} [g(z)-g(0)-D_x g(0) h(z)] \, F(dz) +  p b + \frac{1}{2} \tr[q c] \bigg\}. \label{eq:G^k}
\end{multline}
In the remainder of this section, we verify that (C1)--(C9) hold for this choice of $G^\kappa$ and $G$ as in~\eqref{eq:PIDE-G}, which will complete the proof of Proposition~\ref{pr:comparisonPIDE}. To simplify the notation, we assume that $h$ is the canonical truncation function
\[
  h(z)=z\1_{|z|\leq 1}.
\]
This entails no loss of generality because the PIDE~\eqref{eq:PIDE} does not depend on the choice of $h$.

\begin{lemma}\label{le:C1}
The function $G$ of~\eqref{eq:PIDE-G} satisfies (C1)--(C3).
\end{lemma}
\begin{proof}
  Conditions (C2) and (C3) follow directly from the definitions; we focus on (C1).
  In view of \eqref{eq:intcondRestated}, we may fix $N>1$ and estimate
  \[%
    \ |G(p_k,q_k,f_k(t_k,x_k+\cdot)) - G(p,q,f(t,x+\cdot))| \leq I^1_k + I^2_k  + I^3_{k} + I^4_{k,N} + I^5_{k,N},
  \]%
  where
  \[
    I^1_k =\sup_{(b,c,F) \in \Theta} |b|\,|p_k-p|,\quad\quad I^2_k = \frac{1}{2} \sup_{(b,c,F) \in \Theta} |c|\,|q_k-q|,
  \]
  \begin{multline*}
  I^3_k =   \sup_{(b,c,F) \in \Theta} \bigg\{\int_{|z|\leq 1} \big|\big(f_k(t_k,x_k+z)-f_k(t_k,x_k)-D_x f_k(t_k,x_k) z\big)\\
   - \big(f(t,x+z)-f(t,x)-D_x f(t,x) z\big)\big| \, F(dz)
   \bigg\},
  \end{multline*}
  \begin{multline*}
  I^4_{k,N} =   \sup_{(b,c,F) \in \Theta} \bigg\{\int_{1\leq |z|\leq N} \big|\big(f_k(t_k,x_k+z)-f(t,x+z)\big) \\
    - \big(f_k(t_k,x_k)-f(t,x) \big)\big| \, F(dz)\bigg\},
  \end{multline*}
  \begin{multline*}
  I^5_{k,N} =   \sup_{(b,c,F) \in \Theta} \bigg\{\int_{|z|> N} \big|\big(f_k(t_k,x_k+z)-f(t,x+z)\big) \\
    - \big(f_k(t_k,x_k)-f(t,x) \big)\big| \, F(dz)\bigg\}.
  \end{multline*}
  In view of the assumptions made in (C1) and~\eqref{eq:intcondRestated}, we see that $I^1_k+I^2_k\to0$ as $k\to\infty$. By a Taylor expansion, there are $\xi_{k,z}, \xi_z \in \{|z| \leq 1\}$ such that
  \begin{equation*}
    I^3_k= \sup_{(b,c,F) \in \Theta} \bigg\{\int_{|z|\leq 1} \frac{1}{2} \tr\big[\big(D^2_{xx}f_k(t_k,x_k+\xi_{k,z})-D^2_{xx}f(t,x+\xi_z)\big) zz^\top\big] \, F(dz).
  \end{equation*}
  Using \eqref{eq:intcondRestated} and the locally uniform convergence of $D^2_{xx}f_k$ to $D^2_{xx}f$, it follows that $I^3_k\to0$.
  Similarly, there exist $\xi_{k,z}, \xi_{z} \in \{|z|\leq N\}$ such that
  \begin{equation*}
  I^4_{k,N}= \sup_{(b,c,F)\in \Theta} \bigg\{\int_{1\leq |z|\leq N} \big|\big(D_x f_k(t_k,x_k+\xi_{k,z})- D_x f(t,x+\xi_{z})\big) z\big| \, F(dz) \bigg\},
  \end{equation*}
  and the locally uniform convergence of $D_{x}f_k$ to $D_{x}f$ yields that $I^4_{k,N}\to0$ for any fixed $N$.
  Using the uniform bound on $(f_k)_k$ assumed in (C1), we also see that
 \begin{equation*}
   I^5_{k,N}\leq C \, \sup_{(b,c,F) \in \Theta} \bigg\{\int_{|z| > N} 1 \, F(dz)\bigg\} \leq \frac{C}{N} \, \sup_{(b,c,F) \in \Theta} \bigg\{\int_{|z| > 1} |z| \, F(dz)\bigg\};
 \end{equation*}
 note that the right-hand side is independent of $k$ and finite by~\eqref{eq:intcondRestated}.
 Summarizing the above, we have
 \begin{align*}
   \limsup_{k \to \infty} |G(p_k,q_k,f_k(t_k,x_k+\cdot)) - G(p,q,f(t,x+\cdot))| \leq C/N
 \end{align*}
 for every $N>1$ and the result follows.
\end{proof}
\begin{lemma} \label{le:C4to7}
The functions $G$ of \eqref{eq:PIDE-G} and $(G^\kappa)_{\kappa \in (0,1)}$ of \eqref{eq:G^k} satisfy (C4)--(C7).
\end{lemma}

\begin{proof}
  Conditions (C4)--(C6) follow directly from the definitions of $G$, $G^\kappa$ and~\eqref{eq:intcondRestated}. The proof of (C7) is similar to the verification of~(C1) and therefore omitted.
\end{proof}

\begin{lemma}\label{le:51}
The functions $(G^\kappa)_{\kappa \in (0,1)}$ of \eqref{eq:G^k} satisfy (C8) and (C9).
\end{lemma}
\begin{proof}
  We first show (C8). By definition, we have
  \begin{align*}
  & \ |G^\kappa(p_1,q_1,f(t,\cdot)+ \psi(\cdot), g(t,\cdot)+ \psi(\cdot))
  - G^\kappa(p_2,q_2,f(t,\cdot), g(t,\cdot))|\\
  \leq & \sup_{(b,c,F) \in \Theta} |b|\,|p_1-p_2|
   + \frac{1}{2} \sup_{(b,c,F) \in \Theta} |c|\,|q_1-q_2| + I_1 + I_2,
  \end{align*}
  where
  \begin{align*}
    I_1 & = \sup_{(b,c,F) \in \Theta} \bigg\{\int_{|z|\leq 1} |\psi(z)-\psi(0)-D_x \psi(0) z| \, F(dz)\bigg\}, \\
    I_2 & = \sup_{(b,c, F) \in \Theta} \bigg\{\int_{|z|>1} |\psi(z)-\psi(0)| \, F(dz)\bigg\}.
  \end{align*}
  By a Taylor expansion, we see that there are $\xi_z \in \R^d$ such that
  \begin{align*}
  I_1 &=  \frac{1}{2}\,\sup_{(b,c,F) \in \Theta} \bigg\{\int_{|z|\leq 1} |\tr[D^2_{xx} \psi(\xi_z) \, zz^\top]| \, F(dz)\bigg\}\\
  & \leq \frac{1}{2} \,\|D^2_{xx}\psi\|_\infty \, \sup_{(b,c,F) \in \Theta} \bigg\{\int_{|z|\leq 1} |z|^2 \, F(dz)\bigg\}
  \end{align*}
  and the integral on the right-hand side is bounded by $\cK$ due to~\eqref{eq:intcondRestated}. Similarly,
  \begin{align*}
  I_2 &= \sup_{(b,c,F) \in \Theta} \bigg\{\int_{|z|>1} |D_x\psi(\xi_z) z| \, F(dz) \bigg\} \\
  & \leq \|D_{x}\psi\|_\infty \, \sup_{(b,c,F) \in \Theta} \bigg\{\int_{|z|> 1} |z| \, F(dz)\bigg\}
  \end{align*}
  and again the integral is bounded by $\cK$. Property~(C8) follows, with the constant being~$\cK$ up to a numerical factor.

  The assumptions in~(C9) imply that
  \begin{align*}
  & \ |G^\kappa(p,q, f(t,x+ \cdot), g_1(t,x+\cdot))-
  G^\kappa(p,q, f(t,x+ \cdot), g_2(t,x+\cdot))|\\
  \leq & \ \sup_{(b,c,F) \in \Theta} \bigg\{\Big|\int_{|z|\leq \kappa} g_1(t,x+z)-g_1(t,x)-D_x g_1(t,x) z \, F(dz) \\
  & \phantom{ \ \sup_{(b,c,F) \in \Theta} \bigg\{\Big|\int_{|z|\leq \kappa}} - \int_{|z|\leq \kappa} g_2(t,x+z)-g_2(t,x)-D_x g_2(t,x) z \, F(dz) \Big|\bigg\}.
  \end{align*}
  If $K\subset \R^d$ is the closed ball of unit radius around $x$, a Taylor expansion shows that the above expression is bounded by
  \begin{align*}
   \frac{1}{2}\big(\|D^2_{xx}\,g_1(t,\cdot)\|_K + \|D^2_{xx}\,g_2(t,\cdot)\|_K\big) \sup_{(b,c,F) \in \Theta} \bigg\{\int_{|z|\leq \kappa} |z|^2 \, F(dz) \bigg\},
  \end{align*}
  where $\|\cdot\|_K$ is the uniform norm on $K$. Thus, the claim follows from~\eqref{eq:limitcondRestated}.
\end{proof}

\newcommand{\dummy}[1]{}

\end{document}